\documentclass[a4paper]{amsart}

\usepackage{amsmath,amsfonts,amssymb,amsthm}
\usepackage[colorlinks=true]{hyperref}
\usepackage{graphicx}
\usepackage[all]{xy}
\usepackage{mathrsfs}
\usepackage{enumerate}
\usepackage{oldgerm}
\usepackage{MnSymbol}
\usepackage{stmaryrd}

\theoremstyle{definition}
\newtheorem{theorem}{Theorem}[subsection]

\newtheorem*{theorem*}{Theorem}

\newtheorem{corollary}[theorem]{Corollary}
\newtheorem{definition}[theorem]{Definition}
\newtheorem{examples}[theorem]{Examples}

\newtheorem{example}[theorem]{Example}

\newtheorem{lemma}[theorem]{Lemma}
\newtheorem{proposition}[theorem]{Proposition}
\newtheorem{remark}[theorem]{Remark}
\numberwithin{equation}{section}

\newcommand{\toto}{\rightrightarrows}
\newcommand{\xto}{\xrightarrow}
\newcommand{\from}{\leftarrow}
\newcommand{\xfrom}{\xleftarrow}

\newcommand{\dashto}{\dashrightarrow}
\def\<{\langle}
\def\>{\rangle}
\newcommand{\action}{\curvearrowright}
\newcommand{\raction}{\curvearrowleft}

\newcommand{\R}{\mathbb R}

\newcommand{\Z}{\mathbb Z}
\newcommand{\E}{\mathcal E}

\newcommand{\Ss}{\mathbb S}

\newcommand{\id}{{\rm id}}

\newcommand{\mon}{{\rm Mon}}

\newcommand{\eps}{\varepsilon}

\renewcommand{\d}{{\rm d}}
\newcommand{\G}[1]{G^{(#1)}}
\newcommand{\U}[1]{U^{(#1)}}
\newcommand{\V}[1]{V^{(#1)}}
\newcommand{\e}[1]{\eta^{(#1)}}

\newcommand{\al}{\alpha}
\newcommand{\be}{\beta}

\newcounter{comments}

\begin{document}

\title[Riemannian metrics on differentiable stacks]{
Riemannian metrics on differentiable stacks}

\author{Matias del Hoyo}
\address{Departamento de Geometria, Universidade Federal Fluminense,
Rua Mario Santos Braga, S/N, Campus do Valonguinho, CEP 24020-140, Niteroi, RJ, Brasil.}
\email{mldelhoyo@id.uff.br}

\author{Rui Loja Fernandes}
\address{Department of Mathematics, University of Illinois at Urbana-Champaign, 1409 W. Green Street, Urbana, IL 61801, USA} 
\email{ruiloja@illinois.edu}

\thanks{MdH was partially supported by the ERC Starting Grant No. 279729. RLF was partially supported by NSF grants DMS 1308472, DMS 1405671, DMS 1710884 and FCT/Portugal. Both authors acknowledge the support of the \emph{Ci\^encias Sem Fronteiras} grant 401817/2013-0.}

\date{\today}

\begin{abstract} 
We study Riemannian metrics on Lie groupoids in the relative setting. We show that any split fibration between proper groupoids can be made Riemannian, and we use these metrics to linearize proper groupoid fibrations. As an application, we derive rigidity theorems for Lie groupoids, which unify, simplify and improve similar results for classic geometries. Then we establish the Morita invariance for our metrics, introduce a notion for metrics on stacks, and use them to construct stacky tubular neighborhoods and to prove a stacky Ehresmann theorem.
\end{abstract}

\maketitle

\vspace{-1cm}

\setcounter{tocdepth}{1} 
\tableofcontents 

\vspace{-1cm}


\section{Introduction}


Lie groupoids and differentiable stacks provide a framework to perform differential geometry on singular spaces. The interaction between these theories is twofold. 
Lie groupoids generalize group actions, fibrations, and foliations \cite{mm}, and their transverse geometry is encoded by differentiable stacks. Differentiable stacks, an incarnation of Grothendieck's ideas, first developed in algebraic geometry, give rise to Lie groupoids when endowed with a presentation (\cite{bx,metzler}). For example, orbifolds are the stacks corresponding to proper, \'etale groupoids (\cite{lerman,mm}).


Many results and techniques for smooth manifolds have been succesfully extended to orbifolds, with applications to Poisson geometry, non-commutative geometry and mathematical physics. For instance, Riemannian metrics on orbifolds have allowed important generalizations of the Hodge Theorem or the Gauss-Bonnet formula. However, little is known when dealing with general differentiable stacks, on which the dimension of the isotropy may vary. In particular, a notion of Riemannian metrics for general differentiable stack has been missing.


Recent developments in the theory of Lie groupoids allow us to shed some light on the Riemannian geometry of differentiable stacks. 
In \cite{dhf} we have constructed compatible Riemannian metrics on proper Lie groupoids, called 2-metrics, and used them to linearize groupoids using exponential maps, improving Weinstein-Zung Linearization Theorem and its generalizations considerably (\cite{weinstein,zung,cs}). 
Here we develop the relative version of our theory, constructing compatible metrics on fibrations of Lie groupoids, proving linearization results for maps, and presenting two major applications, namely, (i) rigidity theorems for Lie groupoids, and (ii) a Morita invariance of 2-metrics, leading to a notion of metric on differentiable stacks. 


A special kind of Lie groupoid map, playing a key role here, is a smooth version of the notion of fibration between categories, introduced by Grothendieck. Our first main result concerns the construction of 2-metrics adapted to a fibration, based on the so called \emph{gauge trick}, an averaging argument introduced in \cite{dhf}: 

\begin{theorem*}[Existence of fibered groupoid metrics]
Let $\phi:\tilde G\to G$ be a split fibration between proper Lie groupoids. Then there exist 2-metrics $\tilde\eta$ on $\tilde G$ and $\eta$ on $G$ making $\phi$ into a Riemannian submersion.
\end{theorem*}



Our second main result combines the methods developed by us in \cite{dhf}, to linearize Riemannian groupoids by means of exponential maps, with a thorough analysis on the structure of Lie groupoid fibrations and cleavages. The outcome is a linearization result for fibrations between proper Lie groupoids.
Given $G\toto M$ a Lie groupoid and given $S\subset M$ an embedded saturated submanifold, we denote by $G_S=s^{-1}(S)\toto S$ the restriction to $S$, that is itself a Lie groupoid.


\begin{theorem*}[Linearization of proper groupoid fibrations]
Let $\phi:(\tilde G,\tilde \eta)\to (G,\eta)$ be a fibration between proper Lie groupoids, let $S\subset M$ be an embedded saturated submanifold, and set $\tilde S=\phi^{-1}(S)$.
Then $\phi$ is linearizable around $S$, namely there are linearizations $\tilde\alpha,\alpha$ of $\tilde G,G$ around $\tilde S,S$ making the diagram commutative:
$$\xymatrix@C=50pt{
\nu(\tilde G_{\tilde S}) \supset
\tilde G_{\tilde U} \ar[d]_{\overline{\d\phi}} \ar[r]^{\tilde\alpha}
& \tilde G_{\tilde V} \subset \tilde G \ar[d]^\phi \\
\nu(G_S) \supset G_U \ar[r]^{\alpha} & G_V \subset G}$$
Moreover, if $\phi$ is proper, we can take $\tilde U=\phi^{-1}(U)$ to be an open tube around $\tilde S$.
\end{theorem*} 


The above results are categorifications of basic facts on differential geometry, and at the same time, they are quite powerful and have deep, interesting, geometric consequences. We include in this paper two major applications. First, we give simple geometric proofs of rigidity results on Lie groupoids, generalizing various theorems for classical geometries, such as groups (\cite{nr}), group actions (\cite{palais}) and foliations (\cite{hamilton,rosenberg1}).
They were obtained recently and independently in \cite{cms}, where the authors use instead a theory of deformation cohomology. Our main result in this direction is the following:

\begin{theorem*}[Rigidity of compact Lie groupoids]
Every proper deformation of a compact Lie groupoid $G\toto M$ is trivial.
\end{theorem*}


Our second application is a form of Morita invariance for our metrics. 
Every Lie groupoid $G\toto M$ determines an \emph{orbit stack} $[M/G]$ with a presentation $M\to [M/G]$, and two groupoids determine the same differentiable stack if and only if they are Morita equivalent. A Morita equivalence can always be realized by a fraction of Morita fibrations, an adaptation of Verdier's hypercovers \cite[V.7]{agv}. Pulling back and pushing forward 2-metrics along Morita fibrations we get:


\begin{theorem*}[Morita invariance of groupoid metrics]
If two Lie groupoids are Morita equivalent and one admits a 2-metric, then so does the other. 
\end{theorem*}

We will introduce a notion of equivalence of 2-metrics on a given Lie groupoid, and we will show that a Morita equivalence yields a 1-1 correspondence between equivalence classes of 2-metrics.
We thus obtain a notion of \emph{Riemannian metric on a differentiable stack}, generalizing the usual concepts of Riemannian metrics on manifolds and orbifolds. This opens up the possibility to extend to singular spaces modeled by stacks classical concepts in Riemannian geometry such as geodesics and curvature, as well as classical results such as Hopf-Rinow, Bonnet, etc. Some of these questions will be addressed in \cite{dhdm}. 
Here, as a first application of our stacky metrics, we build stacky tubular neighborhoods and, after developing all the required geometric notions for differentiable stacks, we prove the following:

\begin{theorem*}[Stacky Ehresmann]
Let $\phi:[\tilde M/\tilde G]\to [M/G]$ be a surjective submersion between separated stacks, and let $[S/G_S]\subset [M/G]$ be an embedded substack. Then $\phi$ is linearizable around $[S/G_S]$, and if $\phi$ is proper, the linearization can be made over an open tube $\phi^{-1}([U/G_U])$ around $\phi^{-1}([S/G_S])$.
\end{theorem*}


This last result, in an appropriate sense, generalizes and encompasses all the above linearization and rigidity results for groupoids and classic geometries. It can be thought of as a categorification of a basic result. In fact, one can interpret our usage of metrics to prove it as a categorifying the proof of that basic result.


\vskip 5 pt

\noindent{\bf  Organization.} In Section \ref{sec:fibrations} we review basic facts about fibrations between Lie groupoids and their cleavages. In Section \ref{sec:riemannian:subm} we recall the main properties of Riemannian groupoids and we adapt the gauge trick construction of \cite{dhf} to the case of fibrations, proving that split fibrations between proper groupoids can be made Riemannian. In Section \ref{sec:linear:fibrations} we prove our linearization theorems for Riemannian groupoid submersions and proper groupoid fibrations. In Section \ref{section:rigidity} we give applications of our linearization results to deduce rigidity theorems for Lie groupoids. Finally, in section \ref{sec:stack}, we present a stacky version of our theory, developing fundamental notions such as stacky immersion and submersions, proving the Morita invariance of our metrics, and establishing a stacky Ehresmann theorem. 
\vskip 5 pt

\noindent{\bf  Acknowledgments. } We are grateful to IMPA, UU and UIUC for hosting us at several stages of this project. We thank H.~Bursztyn, E.~Lerman, I.~Marcut and I.~Moerdijk for fruitful discussions, and to M.~Crainic, J.N.~Mestre and I.~Struchiner for sharing with us a preliminary version of their preprint \cite{cms}. We also thank the referee for his comments and suggestions, that helped improve this manuscript.


%
%
%

\section{Lie groupoids fibrations}
\label{sec:fibrations}


We review here definitions and basic facts on fibrations between Lie groupoids, provide examples, and discuss the key concept of a cleavage, that play a role in a sense analogous to connections in ordinary differential geometry. Then we review the correspondence between split fibrations and semi-direct product. Finally, we discuss basic facts on proper maps between Lie groupoids, to be used later.

\subsection{Definitions and examples}




Given a Lie groupoid $G=(G\toto M)$, we denote by $s,t,m,u,i$ its maps of source, target, multiplication, unit and inverses, and by $G^{(k)}$ the manifold of $k$-tuples of composable arrows. In particular, $G^{(1)}=G$ are the arrows and $G^{(0)}=M$ are the objects. By a {\bf map of Lie groupoids} $\phi:G\to H$ we mean a smooth functor, and we denote by $\phi^{(k)}:G^{(k)}\to H^{(k)}$ the induced smooth map. By convention, all our manifolds (including all our groupoids) are second countable and Hausdorff.


\begin{definition}
A map of Lie groupoids $\phi:\tilde G\to G$ is called a {\bf fibration} if both $\phi^{(0)}:\tilde M\to M$ and
$\hat\phi:\tilde G\to G\times_{M}\tilde M$, $\hat\phi(g)=(\phi^{(1)}(g),s(g))$,
are surjective submersions. 
We call $G$ the {\bf base groupoid} and $\tilde G$ the {\bf total groupoid}. 
\end{definition}


\begin{examples}\label{ex:fibrations}\

\begin{enumerate}
 \item 
 A fibration between manifolds $\phi:\tilde M\to M$, regarded as groupoids with only identity arrows, is a surjective submersion.
 A fibration between Lie groups $\phi:\tilde G\to G$, regarded as groupoids with only one object, is a surjective homomorphism. 

 \item By a {\bf family of Lie groupoids} parametrized by a manifold $M$ we mean a fibration $\phi:(\tilde G\toto \tilde M)\to (M\toto M)$ with base $M$. This amounts to give a surjective submersion $\tilde M\to M$ that is constant along the $\tilde G$-orbits. 
 
 \item If $G\toto M$ is a Lie groupoid acting along a surjective submersion $\mu:{\tilde M}\to M$, the projection from the corresponding action groupoid
 $(G\ltimes {\tilde M}\toto {\tilde M})\to (G\toto M)$ is a fibration. 
 Every fibration for which $\hat\phi$ is bijective arises in this way -- such a $\phi$ is called an {\bf action fibration}.
 
 \item If $G\toto M$ is a Lie groupoid, then the projections $TG\to G$ and $T^*G\to G$ are fibrations of the tangent and cotangent groupoids over $G$. More generally, for any VB-groupoid $\Gamma\to G$ (see, e.g., \cite{bcdh}) the projection is a fibration. 
 
 \item Morita maps play an important role in the theory of differentiable stacks. We will study them in section \ref{sec:stack}. If $\phi:\tilde G\to G$ is a Morita map such that $\phi^{(0)}$ is a surjective submersion, then $\phi$ is a fibration. 
\end{enumerate}
\end{examples}


If $\phi:\tilde G\to G$ is a fibration then $\phi^{(k)}$ is a submersion for every $k$. Moreover, by a standard transversality criterion (cf. \cite[A]{bcdh}), the groupoid-theoretic {\bf fiber product} of $\phi$ and any other groupoid map exists and is well-behaved with respect to the topologies and the tangent spaces. In such a fiber product,
$$\xymatrix@R=15pt{\ar@{}[dr]|{\rm fp}
\tilde H \ar[r] \ar[d]_{\psi} & \tilde G \ar[d]^{\phi} \\ H \ar[r] & G
}$$
if $\phi$ is a fibration, then the same holds for its {\bf base-change} $\psi$.



\begin{definition}
Given a fibration $\phi:\tilde G\to G$, 
its {\bf kernel} $K\toto\tilde M$ and its {\bf fibers} $\tilde G_x\toto\tilde M_x$ are the fiber products between $\phi$ and the groupoid inclusions $M\to G$ and $\id_x\to G$, respectively.
\end{definition}

The kernel $K$ consists of the arrows in $\tilde G$ that are mapped into identities. It is a family of Lie groupoids, the fibers, parametrized by the base. We remark that $\phi$ is an action fibration if and only if $K=\tilde M$, and $\phi$ is a Morita fibration if and only if $K=\tilde M\times_M\tilde M$.
In a fibration $\phi:\tilde G\to G$ we can think of the total groupoid $\tilde G$ as an extension of the base $G$ by the kernel $K$,
and visualize $\tilde G$ as sitting over $G$, $\phi$ as a projection, and $K$ as the vertical arrows.

\begin{definition}
A {\bf cleavage} $\sigma$ for a fibration $\phi:\tilde G\to G$ is a smooth section for the map $\hat\phi:\tilde G\to G\times_M\tilde M$.
The cleavage is {\bf unital} if it preserves identities, namely $\sigma(\id_{\phi(\tilde x)},\tilde x)=\id_{\tilde x}$, and is  {\bf flat} if it is closed under multiplication, namely
$\sigma(g_2g_1,\tilde x)=\sigma(g_2,\tilde{t}(\sigma(g_1,\tilde x)))\sigma(g_1,\tilde x)$. 
\end{definition}

One can think of a cleavage as a choice of horizontal arrows which allow us to relate different fibers, playing a role similar to a connection in differential geometry.
A fibration endowed with a unital flat cleavage is called a {\bf split fibration}.

\begin{examples}\label{ex:fibrations:cleavages}\

\begin{enumerate}
\item If $\phi:\tilde G\to G$ is a fibration between Lie groups, then a cleavage $\sigma$ is just a smooth section, and is unital and flat if and only if it is a morphism. Therefore, $\phi$ splits if and only if $\tilde G$ is a semi-direct product $G\ltimes K$. 

\item A family of Lie groupoids $\phi:\tilde G\to M$ is a split fibration, for the unit map $\tilde u:\tilde M\to \tilde G$ yields a unital, flat, cleavage in an obvious way. 

\item In an action fibration $(\tilde G\toto {\tilde M})\to (G\toto M)$ the map $\hat\phi$ is invertible, hence there exists a unique cleavage, which is both unital and flat.

\item Cleavages for the tangent bundle projection $TG\to G$ that are both linear and unital are called {\bf connections} of $G$ in \cite{ac,dhf}. Every Lie groupoid admits a connection, but in general there is not a flat one. 
Linear cleavages for a general VB-groupoid $\Gamma\to G$ are called {\bf horizontal lifts} in \cite{bcdh}.

\item For an example on which does not exists a cleavage, consider the Morita fibration $\phi:(\R\times\R\toto\R)\to(S^1\times S^1\toto S^1)$ induced by the standard covering map $\R\to S^1$.
\end{enumerate}
\end{examples}


Given a cleavage $\sigma$ and \smash{$y\xfrom gx\in G$}, we can define a {\bf base-change functor} between the fibers by parallel transport: 
$$\theta_g:\tilde G_x\to\tilde G_y \qquad
\theta_g(\tilde x)=\tilde t\sigma(g,\tilde x)\quad
\theta_g(\tilde x_2\xfrom k\tilde x_1)= \sigma(g,\tilde x_2)k\sigma(g,\tilde x_1)^{-1}$$
This defines {\bf quasi-actions} $\theta:G\tilde\action \tilde M$ and 
$\theta:G\tilde\action K$, for $g\mapsto\theta_g$ does not preserve, in general, identities nor compositions. It does so if and only if $\sigma$ is unital and flat.

\subsection{Split fibrations and semidirect product}
\label{section:split-fibration}

In a split fibration $\phi:\tilde G\to G$
the horizontal arrows define a wide Lie subgroupoid, the {\bf horizontal groupoid}
$$(H \toto \tilde M)\subset (\tilde G\toto \tilde M),$$
and the base-change functors define a groupoid action of the base over the kernel, in the sense of definition 
(\cite[Def. 2.5.1]{mkbook}), that we recall now.

\begin{definition}
Let $G\toto M$ be a Lie groupoid, and let $q:(K\toto \tilde M)\to M$ be a family of Lie groupoids parametrized by $M$. Then an {\bf action}
$$\theta:(G\toto M)\action (K\toto \tilde M)$$
consists of a pair of groupoid actions $\theta:G\action K$, $\theta:G\action \tilde M$, in such a way that the structural maps of $K\toto \tilde M$ are $G$-equivariant.
\end{definition}


The whole split fibration can be recovered out of the base, the kernel and the base-change action. For the inverse procedure, i.e., to construct a split fibration out of an action, we have the following generalization from groups to groupoids of the notion of semi-direct product (see, e.g., \cite{mkbook}):


\begin{definition}
Given an action $\theta:(G\toto M)\action (K\toto \tilde M)$, the {\bf semidirect product} $G\ltimes K =(G\times_M K \toto \tilde M)$ is the Lie groupoid where:
\begin{itemize}
\item the arrows consist of pairs $(g,k)$ such that $q(k)=s(g)$; 
\item the source and target are given by $s(g,k)=s(k)$, $t(g,k)=\theta_{g}(t(k))$; 
\item the multiplication is given by $(g',k')(g,k)=(g'g,\theta_{g^{-1}}(k')k)$.
\end{itemize}
\end{definition}

In the semi-direct product $G\ltimes K$, the obvious projection $(g,k)\mapsto g$ defines a split fibration 
$(G\times_M K \toto \tilde M)\to (G\toto M)$,
with a canonical unital flat cleavage given by $\sigma(g,\tilde x)=(g,\id_{\tilde x})$.
These two constructions are mutually inverse:

\begin{theorem}[{\cite[Thm 2.5.3]{mkbook}}]
There is a 1-1 correspondence between (isomorphisms classes of) split fibrations and actions of a groupoid over another one. Every split fibration is isomorphic to a semi-direct product.
\end{theorem}

\begin{remark}
This can be seen as a smooth version of the correspondence between categorical split fibrations $\phi:\tilde G\to G$ and functors $G^\circ\to\{\text{Categories}\}$,
which extends to a correspondence between arbitrary fibrations and pseudo-functors (see, e.g., \cite[Sec.~3.1]{vistoli})).
A smooth correspondence for arbitrary fibrations is yet to be explored. 
\end{remark}



Recall that given $\phi_1:G_1\to H$ and $\phi_2:G_2\to H$ Lie groupoid maps, their {\bf homotopy fiber product} $G_1\tilde\times_H G_2$ (see, e.g., \cite{dh, mm}) is the Lie groupoid of triples \smash{$(x_1,\phi_1(x_1)\xfrom h\phi_2(x_2),x_2)$}, where an arrow between two such objects is a pair of arrows in $G_1\times G_2$ inducing a commutative square in $H$.  The homotopy fiber product fits into a universal square, commutative up to isomorphism.


Every Lie groupoid map $\phi:G\to H$ has a {\bf canonical factorization},
a formal analog to the path fibration used in topology to show that any map is a fibration up to homotopy:
$$\xymatrix@R=10pt{
 & G' \ar[dr]^{\phi'} 
 \\ G \ar[ur]^{\iota} \ar[rr]_\phi & & H
}$$
Here $G'$ is the homotopy fiber product $H\tilde\times_{H} G$ over $\id_H$ and $\phi$, which always exists, the maps $\iota$ and $\phi'$ are given, on objects, by 
\smash{$\iota:x\mapsto (\phi(x),\phi(x)\xfrom\id \phi(x),x)$} and \smash{$\phi':(y,y\xfrom h \phi(x),x)\mapsto y$}, and they are extended to arrows in the obvious way. It is clear that $\iota$ is an embedding and a categorical equivalence.


\begin{lemma}\label{lemma:associated-fibration}
If $\phi:G\to H$ is a fibration then $\phi':G'\to H$ is a split fibration with a canonical cleavage. A unital cleavage $\sigma$ for $\phi$ is the same as a Lie groupoid retraction $r:G'\to G$ such that $r\iota=\id_G$ and $\phi r=\phi'$.
\end{lemma}

\begin{proof}
By assumption $\phi^{(0)}:M\to N$ is a surjective submersion, then so does the base-change $\pi_1:M'=H\times_NM\to H$, and the composition $\phi'^{(0)}=t\pi_1$. Let us show that $G'\to H\times_NM'$ is a surjective submersion as well. We can identify $G'\cong H\times_N H\times_N G$ by encoding an arrow
$$(y_2,y_2\xfrom {h_2} \phi(x_2),x_2)\from
(y_1,y_1\xfrom {h_1} \phi(x_1),x_1)$$
as the triple $(h,h_1,g)$. Under this identificaiton, $G'\to H\times_NM'=H\times_NH\times_NM$ is just $(h,h_1,g)\mapsto (h,h_1,s(g))$, which is a submersion, with global section $(h,h_1,x)\mapsto(h,h_1,\id_x)$.
It follows that $\phi'$ is a fibration, and this section is a unital flat cleavage.
The second statement, identifying cleavages of $\phi$ with retractions $r:G'\to G$ is proven set-theoretically in \cite[Prop. 2.2.3]{dh0}, and its adaptation to the smooth setting is straightforward.
\end{proof}

\subsection{Fibrations and properness}


Let us recall that an (Hausdorff) groupoid $G\toto M$ is called {\bf proper} if the anchor $(t,s):G\to M\times M$ is a proper map. Concerning properness, for a split fibration, we have:

\begin{lemma}\label{lemma:proper-semidirect-product}
Let $\phi:\tilde G\to G$ be a split fibration with kernel $K$ and horizontal groupoid $H$.
\begin{enumerate}[(a)]
 \item If $G$ and $K$ are proper, then $\tilde G$ is proper,
 \item If $\tilde G$ is proper then both $H$ and $K$ are proper.
\end{enumerate}
\end{lemma}

\begin{proof}
We can assume that $\tilde G = G\times_M K$. To prove (a), denoting by $\tilde{s}$ and $\tilde{t}$ the source and target of $\tilde G$, we find that if $C\subset \tilde M\times \tilde M$ then:
\[ (\tilde{s}\times\tilde{t})^{-1}(C)\subset (G\times_M K)\cap \left( (s_G\times t_G)^{-1}(\phi^{(0)}\times\phi^{(0)})(C))\times(s_K\times t_K)^{-1}(C)\right)\]
If $C$ is compact, then the right-hand side is clearly compact, so we must have that $(\tilde{s}\times\tilde{t})^{-1}(C)$ is compact. This shows that $\tilde G=G\times_M K$ is proper.

We can prove (b) by using that a closed subgroupoid of a proper groupoid is proper. Since a section of a surjective submersion is a closed map, the set of horizontal arrows $H\subset \tilde G$ is a closed submanifold. Similarly, $M\subset G$ is a closed submanifold, and therefore \smash{$K={\phi^{(1)}}^{-1}(M)\subset \tilde G$} is also closed.
\end{proof}


\begin{example}
Both implications in the previous lemma are, in general, strict.
For instance, the action by translations $(\R\toto\ast)\action(\R\toto\R)$ leads to a split fibration with proper total groupoid but whose base is not proper. 
Besides, translations on the units and arrows of the pair groupoid $(\R\toto\ast)\action(\R\times\R\toto\R)$ yields a fibration for which kernel and horizontal groupoids are proper but the total groupoid is not.
\end{example}


Besides proper groupoids, we are interested in proper maps between groupoids, which play a crucial role in one of our main results, namely the groupoid version of Ehresmann's theorem. We propose the following definition:

\begin{definition}\label{def:proper-map}
A map between Lie groupoids $\phi:\tilde G\to G$ is {\bf proper} if $\phi^{(1)}$ is, or equivalently, if $\phi^{(k)}$ is proper for some, and hence all, $k\geq 1$.
\end{definition}


In a proper fibration, the total groupoid is proper if and only if the base groupoid is proper. 
Also, one can show that a fibration is proper if and only if its kernel is a {\bf proper family}, namely a fibration $\phi:({\tilde G\toto \tilde M})\to ({M\toto M})$ that is proper. These statements are easy to prove and left to the reader.

Properness of a map is intimately related to the tube principle, see e.g.~\cite{dh}. Within the context of Lie groupoids, we say that a map $\phi:\tilde G\to G$ satisfies the {\bf tube principle} if for any saturated embedded submanifold $S\subset M$ and any open subgroupoid $\tilde U=(\tilde U_1\toto \tilde U_0)$, $\tilde G_{\tilde S}\subset \tilde U\subset \tilde G$, where $\tilde S=(\phi^{(0)})^{-1}(S)$, there exists an open subgroupoid $G_S\subset V\subset G$ such that $\phi^{-1}(V)\subset \tilde U$.

\begin{proposition}
\label{prop:tube}
Let $\phi:\tilde G\to G$ be a proper submersion. If either (i) $G$ is proper or (ii) $\phi$ is a fibration, then $\phi$ satisfies the tube principle.
\end{proposition}
\begin{proof}
Given $S\subset M$ and an open subgroupoid $\tilde G_{\tilde S}\subset \tilde U\subset \tilde G$, let us set:
$$W:=G\setminus(\phi^{(1)}(\tilde G\setminus\tilde U)).$$ 
Notice that $W$ is an open set, since $\phi^{(1)}$ is a proper map, hence closed.

(i) If $G$ is proper then there is an open $S\subset U\subset M$ such that $G_U\subset W$ (see \cite[Lemma 5.3]{dhf}) and we can take $V=G_U$. To be precise, this follows from the proof rather than from the statement: the requirement of the open neighborhood being a subgroupoid is not used along the proof.

(ii) If $\phi$ is a fibration then we assert that $W$ is already an open subgroupoid, so we can take $V=W$. In fact, given $g,h\in W$ two composable arrows, and given $\tilde k$ an arrow over the product $gh$, we can lift $h$ to an arrow $\tilde h$ with the same source as $\tilde k$. Then both $\tilde h$ and $\tilde k\tilde h^{-1}$ are in $\tilde U$, for its projections are in $W$. Since $\tilde U$ is closed under products it follows that $\tilde k$ also belongs to $\tilde U$. This shows that $gh$ is in $W$. The fact that $W$ is closed under inversion is proved similarly.
\end{proof}


\section{Riemannian submersions}
\label{sec:riemannian:subm}

We recall the definition and main properties of Riemannian groupoids presented in \cite{dhf}, in particular the gauge trick, a recipe to construct metrics on proper groupoids. Then we develop a fibered version of Haar systems, and we adapt the gauge trick to fibrations,
obtaining our first major result: every split fibration between proper groupoids can be made Riemannian.

\subsection{Riemannian groupoids}
\label{sec:review}

Recall that a submersion $p:E\to B$ is {\bf Riemannian} if $E,B$ are manifolds endowed with metrics $\eta^E,\eta^B$ such that $\d_ep:T_eF^\bot\to T_bB$ is an isometry for every $e\in E$. Here $F\subset E$ denotes the fiber. Equivalently, we can require the map $(d_ep)^*:T_b^*B\to T_eF^\circ$ to be an isometry with respect to the dual metrics. 
Given a Riemannian submersion $p:E\to B$, the metric $\eta^E$ is {\bf $p$-transverse}, in the sense that for any two points $e,e'$ in the same fiber $F_b$ the composition $T_eF_b^\circ\cong T_b^*B\cong T_{e'}F^\circ$ is an isometry. Any $p$-transverse metric induces a {\bf push-forward metric} $p_*\eta^E=\eta^B$ on the base. 


Given a Lie groupoid $G\toto M$, and recalling that $\G2$ stands for the pairs of composable arrows, there is a natural group action $S_3\action \G2$ by permuting the vertices of the corresponding commutative triangles. There are also three proper and free commuting groupoid actions $G\action\G2$, given by:
$$k\cdot (g,h)=(kg,h), \qquad k\cdot (g,h)=(gk^{-1},kh), \qquad k\cdot (g,h)=(g,hk^{-1}).$$
The associated principal $G$-bundles are given, respectively, by the face-maps:
\begin{align*} 
\pi_2:\G2\to G, &&m:\G2\to G, &&\pi_1:\G2\to G,\\
(g,h)\mapsto h, &&(g,h)\mapsto gh, &&(g,h)\mapsto g.
\end{align*}


\begin{definition}
A {\bf 2-metric} on $G\toto M$ is a metric $\e2$ on the manifold $\G2$ which is invariant under $S_3$ and that is transverse to one (and hence all) the face-maps $\G2\to G$. The pair $(G\toto M,\e2)$ is called a \textbf{Riemannian groupoid}.
\end{definition}

Equivalently, 2-metrics can be described, and this is how it appears in \cite{dhf}, by saying that they are $G$-invariant. Given a Lie groupoid action over a manifold $G\action E$, a metric $\eta$ on $E$ is $G$-invariant if the normal representations of the action groupoid 
$\tilde\theta_{(g,e)}:T_eO^\circ\cong T_{ge}O^\circ$ are by isometries. Here $O\subset E$ denotes the orbit of $e$. Note that when the action is free and proper, a metric $\eta^E$ is $G$-invariant if and only if it is transverse to the underlying submersion.


It is shown in \cite{dhf} that many classes of Lie groupoids admit 2-metrics, including all proper groupoids. Our construction of metrics on proper groupoids is by averaging  metrics. To do this we choose a connection on $G$ (cf.~Example \ref{ex:fibrations:cleavages}.d), yielding a quasi-action $G\tilde\action S^2(TG^{(2)})$ on symmetric 2-tensors, and a normalized {\bf Haar density} $\mu^G$, that exists by properness. 
We give here a brief outline for later reference. 


\begin{definition}
\label{defn:cotg:average}
Given $\theta:G\action E$ a groupoid action, and $\eta$ a metric on $E$, the {\bf cotangent average} $\underline{\eta}\in \Gamma(E,S^2(T^*E))$ 
is defined by averaging the dual metric:
$$(\underline\eta)^*_e(\alpha,\beta):=I_\theta(\eta^*)_e(\alpha,\beta)= \int_{G(-,x)} \eta^*_{ge}(g\alpha,g\beta) \mu^x(g) $$
where $q:E\to M$ is the moment, $x=q(e)$, $g$ runs over $G(-,x)$, and  $\alpha,\beta\in T_e^*E$.
\end{definition}

One must use cotangent averages rather than tangent averages, as explained in \cite{dhf}, 
to get a {\bf $G$-invariant} metric on $E$, for $G$-invariant metrics on the tangent bundle do not form a convex set. 
Note that when the action $G\action E$ is free and proper, a metric $\eta^E$ is $G$-invariant if and only if it is transverse to the underlying submersion. The fundamental properties of the cotangent average are listed in the following proposition, whose proof is straightforward.

\begin{proposition}\label{prop:averaging-properties}\

\begin{enumerate}[(a)]
\item If $\eta$ is $G$-invariant, then $\eta$ and $\underline\eta$ agree in the directions normal to the orbits.
\item If $p:(E,\eta^E)\to (B,\eta^B)$ is a Riemannian submersion equivariant for actions $\theta^E:G\action E$ and $\theta^B:G\action B$, then $p:(E,\underline\eta^E)\to(B,\underline\eta^B)$ is also Riemannian.
\item If $p:(E,\eta^E)\to (B,\eta^B)$ is a Riemannian submersion and invariant for an action $\theta^E:G\action E$, then $p:(E,\underline\eta^E)\to(B,\eta^B)$ is also Riemannian.
\end{enumerate}
\end{proposition}


We can summarize now our method to construct 2-metrics on proper  groupoids, referred to as the {\em gauge trick} in \cite{dhf}, that will be adapted later to the fibered case:

\begin{remark}[Gauge Trick]\label{recipe}\

\begin{enumerate}[(a)]
 \item One starts by invoking the existence of an $s$-transverse metric $\eta=\eta^{[1]}$ on $G$, since any submersion admits a transverse metric.
 \item Then, one endows the $3$-fold $s$-fiber product $G^{[3]}$ with the {\it fiber product metric} of $\eta^{[1]}$ with itself, denoted $\eta^{[3]}$ (see \cite[Remark 2.5]{dhf}). This is a 2-metric $\eta^{[3]}$ for the submersion groupoid 
 $G^{[2]}\toto G$ arising from $s:G\to M$.
 \item Finally, one replaces $\eta^{[3]}$ by its cotangent average $\underline{\eta^{[3]}}$ with respect to the free proper action
 $$G^{[3]}\raction G \qquad (h_1,h_2,h_3)\cdot g= (h_1g,h_2g,h_3g).$$
The metric $\underline\eta^{[3]}$ is still a 2-metric for $G^{[2]}\toto G$, but in addition it is $G$-invariant, so it descends to a metric $\eta^{(2)}$ on $G^{(2)}=G^{[3]}/G$, that is  a 2-metric on $G\toto M$.
\end{enumerate}
\end{remark}

We know that a 2-metric induces a 1-metric and also a 0-metric. The gauge trick yields out of an $s$-transverse metric $\eta$ on $G$ a 2-metric $\eta^{(2)}$, but also subsidiary metrics that make all maps in the following diagram Riemannian submersions:
$$\xymatrix{
(G^{[3]},\underline\eta^{[3]}) \ar@<0.30pc>[r] \ar[r] \ar@<-0.30pc>[r] \ar[d]& (G^{[2]}, \underline\eta^{[2]})\ar@<0.2pc>[r]\ar@<-0.2pc>[r] \ar[d]& (G,\underline\eta) \ar[d]^t\\
(\G2,\eta^{(2)}) \ar@<0.30pc>[r] \ar[r] \ar@<-0.30pc>[r]& (\G1,\eta^{(1)}) \ar@<0.2pc>[r]\ar@<-0.2pc>[r] & (M,\eta^{(0)}).
}$$
The induced metrics $\underline\eta^{[2]}$ and $\underline\eta$ are exactly the cotangent averages for the obvious right actions (cf.~Proposition \ref{prop:averaging-properties}). In general, we have little control over the resulting metrics for they rely on the choice of connection and Haar system. Even if we start with a 1-metric $\eta$, in general, we have $\underline\eta^{(1)}\ne\eta^{(1)}$. Still, we have:

\begin{lemma}\label{lemma:trick-fix}
If $\eta$ is a 1-metric on $G\toto M$ and $\eta^{(2)}$ is the 2-metric obtained by the gauge trick, then both $\eta$ and $\eta^{(2)}$ induce the same 0-metric on the units.
\end{lemma}

\begin{proof}
This follows from the above diagram: since $\eta$ is $G$-invariant then, by  Proposition \ref{prop:averaging-properties} (c), both $\eta$ and $\underline\eta$ agree in normal directions to the orbits, and therefore they have the same pushforward metric along $t$.
\end{proof}


\subsection{Riemannian submersions}

\label{subsection:riemannian-submersion}


We are interested in maps between Riemannian groupoids which preserve the metrics:

\begin{definition}
A Lie groupoid map $\phi:\tilde G\to G$ is a called a {\bf Riemannian submersion} if both $\tilde G$ and $G$ are endowed with 2-metrics, for which $\phi^{(2)}$ becomes a Riemannian submersion.
\end{definition}

Note that if $\phi^{(2)}$ is a Riemannian submersion then the same holds for $\phi^{(1)}$ and $\phi^{(0)}$. 
Requiring the metrics to be compatible both with the groupoid structure and the submersion may seem too restrictive a priori. However, we will show that such metrics exist when the groupoids are proper and $\phi$ is a split fibration. This is our first major result and yields plenty of examples on which such metrics exist. It will be used later to deduce some fundamental properties of groupoids and stacks.


\begin{theorem}
\label{thm:Riemannian:submersions}
Every split fibration $\phi:\tilde G\to G$ between proper groupoids can be made Riemannian, i.e., there exist 2-metrics $\tilde\eta$ on $\tilde G$ and $\eta$ on $G$ for which $\phi$ is a Riemannian submersion.
\end{theorem}

Equivalently, a semidirect product $G\ltimes K$ of proper groupoids admits a 2-metric such that the projection onto $G$ is a Riemannian submersion. 
We will prove the theorem by adapting the gauge trick from the previous section, now choosing the metric $\eta$ on $\tilde G$ already compatible with the fibration, and performing an averaging with respect to suitable fibered Haar densities and connections. 

%
%
%
%
%

\subsubsection{Fibered Haar densities}


In a semi-direct product groupoid $G\ltimes K$ the source-fiber over $x$ can be written as
$(G\ltimes K)(-,x)\cong G(-,\phi(x)) \times K(-,x)$,
and consequently, we have the following vector bundle isomorphisms involving the Lie algebroids of $G\ltimes K$, $G$, $K$ and $H$:
$$A_{G\ltimes K}\cong \phi^*A_G\oplus A_K\cong A_H\oplus A_K.$$
In fact, it is possible to define fibrations and semi-direct products at the infinitesimal level, and to show that the Lie functor preserves these structures (\cite[\S 4.4-5]{mkbook}), but we do not need to go further than the above vector bundle isomorphisms here.



Recall that a {\bf Haar density} in a Lie groupoid $G\toto M$ is a density $\mu$ on the vector bundle underlying its algebroid $A_G$, which by pullback, yields a smooth family of densities $\{\mu^x\}_{x\in M}$ on the tangent bundles of the source-fibers
$TG(-,x)\cong t^*A$.
The Haar density is {\bf normalized} if ${\rm supp}(\mu^x)$ is compact for all $x\in M$ and
$$\int_{G(-,x)}\mu^x(g)=1.$$
Every proper groupoid admits a normalized Haar density, see \cite[A.2]{dhf} and references therein.
%


Let $\mu_G$ and $\mu_K$ be Haar densities in $G$ and $K$ respectively. Recalling that $A_{G\ltimes K}$ is isomorphic, as a vector bundle, to $\phi^*A_G\oplus A_K$, we can endow $G\ltimes K$ with the {\bf fibered density} $\mu_G\ltimes\mu_K$, defined as the pointwise product between the pullback density of $\mu_G$ and $\mu_K$.

When working with semi-direct products between proper groupoids, the total groupoid is also proper (cf.~Lemma \ref{lemma:proper-semidirect-product}), hence it admits a normalized Haar density. However, in order to apply Fubini's Theorem we do need a fibered density, so the following lemma becomes important:

\begin{lemma}
If $\mu_G$ and $\mu_K$ are normalized, then so does the fibered density $\mu_{\tilde G}=\mu_G\ltimes\mu_K$. 
\end{lemma}

\begin{proof}
If $\tilde x\in \tilde M$, then the support of $(\mu_{\tilde G})^{\tilde x}$ on the $s$-fiber $\tilde G(-,\tilde x)\cong G(-,x)\times K(-,\tilde x)$ identifies with the product of the supports of $\mu_G^x$ and $\mu_K^{\tilde x}$, so it is compact. Moreover, by Fubini's Theorem, we have
\begin{align*}
\int_{\tilde G(-,\tilde x)} \mu_{\tilde G}(g,h) &=
\int_{G(-,x)\times K(-,\tilde x)} \mu_G(g)\times\mu_H(h)\\
&=\int_{G(-,x)}\left[ \int_{K(-,\tilde x)} \mu_G(g)\right] \mu_H(h)=1.
\end{align*}
\end{proof}

\subsubsection{Fibered connections}

Recall that a connection on $G\toto M$ is a linear unital cleavage for the fibration $TG\to G$ (cf. Example \ref{ex:fibrations:cleavages} (d)), or in other words, a vector bundle map $\sigma:s^*TM\to TG$  over $G$ such that $\d s\cdot\sigma=\id$ and $\sigma|_{TM}=\d u$. Connections are important for they allow us to extend a given action $G\action E$ to a quasi-action over the tangent bundle $TE$, and ultimately to average metrics. 

\begin{lemma}
Given $G\ltimes K\xto\phi G$ a split fibration, and given $\sigma^G$ and $\sigma^K$ connections on $G$ and $K$ respectively, 
the formula 
$$\tilde \sigma_{(g,k)}(v)=(\sigma^G_g(\d\phi(v)),\sigma^K_k(v))$$
defines a connection for the semi-direct product $\tilde G$.
\end{lemma}

\begin{proof}
The following is a good fiber product of manifolds, in the sense of \cite[2.2]{dh},
$$\xymatrix{\ar@{}[dr]|{\rm fp}
\tilde G \ar[r]^\pi \ar[d]_\phi & K \ar[d]^{\phi s} \\ G \ar[r]_s & M,
}$$
and therefore it induces a fiber product between their tangent bundles, that can be displayed as a short exact sequence of vector bundles over $\tilde G$:
$$0\to T\tilde G \to \phi^*TG\oplus \pi^* TK \to (\phi s)^*TM\to 0.$$
Hence we can identify $T\tilde G$ with the kernel of the map $(v,w)\mapsto \d s(v) - \d\phi\cdot \d s(w)$. Since
$\d s\cdot\sigma^G_g(\d\phi(v))=\d\phi(v)=\d\phi \cdot\d s\cdot\sigma^K_k(v)$, the above formula does define a vector bundle map $\tilde\sigma:\tilde s^*T\tilde M\to T\tilde G$.
It is easy to check directly from the formula that $\d\tilde s\cdot\tilde\sigma=\id$ and that $\tilde\sigma|_{T\tilde M}=\d u$.
\end{proof}

When working with semi-direct products, we shall always use a connection constructed as in this lemma, which we call a {\bf fibered connection}.

\subsubsection{Fibered averaging}
Given $\phi:\tilde G=G\ltimes K\to G$ a split fibration between proper groupoids, we fix a fibered normalized Haar density $\tilde\mu$ and a fibered connection $\tilde\sigma$, built out of data from $G$ and $K$, as in the previous discussion. The averaged metrics produced with them enjoy the following fundamental property:

\begin{proposition}\label{prop:fiberd-averaging}
Let $p:(E,\eta^E)\to (B,\eta^B)$ be a Riemannian submersion, and assume that $\tilde G\action E$ and $G\action B$ are groupoid actions compatible with the projection $p:E\to B$, in the sense that  
$p((g,k)e)=g(p(e))$ for all $(g,k)\in \tilde{G}=G\ltimes K$.
$$\xymatrix{
\tilde G \ar[d]_\phi \ar@{}|{\action}[r]& E \ar[d]^p\\
G \ar@{}|{\action}[r] & B}$$
Denote by \smash{$\underline\eta^E$} and \smash{$\underline\eta^B$} the cotangent averages of $\eta^E$ and $\eta^B$ with respect to $\tilde G$ and $G$, with averaging data as above. Then \smash{$p:(E,\underline\eta^E)\to(B,\underline\eta^B)$} is also Riemannian.
\end{proposition}

\begin{proof}
We work at the level of cotangent bundles. The map $p:(E,\eta^E)\to (B,\eta^B)$ is a Riemannian submersion  if and only if $(\d_e p)^*:(T^*_{p(e)} B,\eta^B)\to (T^*_eE,\eta^E)$ is an isometric embedding for each $e\in E$. After a choice of fibered connection, the fact that the groupoid actions commute with the projections yields for any $(g,k)\in \tilde{G}=G\ltimes K$ and $\alpha,\beta\in T_{p(e)}^*B$:
\begin{align*}
\mu^E((g,k)\cdot (\d_e p)^*\al,(g,k)\cdot (\d_e p)^*\be)&=\mu^E((\d_{(g,k)e}p)^*(g\cdot \al),(\d_{(g,k)e}p)^*(g\cdot \be))\\
&=\mu^B(g\cdot \al,g\cdot \be).
\end{align*}
Integrating relative to a fibered normalized Haar density $\tilde\mu$, and using Fubini, we conclude that the cotangent average metrics satisfy:
\begin{align*}
\underline{\mu}^E((\d_e p)^*\al, (\d_e p)^*\be)=\underline{\mu}^B(\al,\be).
\end{align*}
In other words, $(\d_e p)^*:(T^*_{p(e)} B,\underline{\eta}^B)\to (T^*_eE,\underline{\eta}^E)$ is an isometric embedding for every $e\in E$, so \smash{$p:(E,\underline\eta^E)\to(B,\underline\eta^B)$} is a Riemannian submersion.
\end{proof}

We are finally ready to prove our first main theorem:


\begin{proof}[Proof of Theorem \ref{thm:Riemannian:submersions}]
We will construct the 2-metrics $\tilde\eta^{(2)}$ on $\tilde G^{(2)}$ and $\eta^{(2)}$ on $G^{(2)}$ by performing the gauge trick \ref{recipe} simultaneously on $\tilde G$ and $G$.
\medskip

\emph{Step 1:} Construction of initial metrics $\tilde\eta$ on $\tilde G$ and $\eta$ on $G$ which are $s$-transverse and compatible with the fibration.
\smallskip

We choose first a metric $\eta$ on $G$ which is $s$-transverse. To construct the metric $\tilde\eta$ on $\tilde G$, we
note that we have the following fiber products of manifolds
$$\xymatrix@R=10pt{\ar@{}[dr]|{\rm fp}\ar@{}[d];[ddr]|{\rm fp}
\tilde G \ar[r] \ar[d] & K \ar[d] \\ H \ar[r] \ar[d] & \tilde M \ar[d] \\ G \ar[r] & M.
}$$
Now chose a metric $\tilde\eta^M$ on $\tilde M$ which is $\phi^{(0)}$-fibered. Then construct a metric $\eta^K$ on $K$ that is transverse to $s:K\to \tilde M$ and induces the previous metric $\tilde\eta^M$ on the base. After that, compute the fiber product metrics first on $H=G\times_M\tilde M$, and then on $\tilde G=H\times_{\tilde M}K$. The resulting metric $\tilde\eta$ will be both $s$-transverse and compatible with the fibration.

\medskip

\emph{Step 2:} Construction of metrics on the 3-fold products $\tilde G^{[3]}$ and $G^{[3]}$.
\smallskip

We can perform this step without any modification: we take the iterated fiber product metrics $\tilde\eta^{[3]}$ and $\eta^{[3]}$ on $\tilde G^{[3]}$ and $G^{[3]}$. These are 2-metrics on the Lie groupoids $\tilde G^{[2]}\toto \tilde G$ and $G^{[2]}\toto G$ respectively, and they make the canonical projection $\tilde G^{[2]}\to G^{[2]}$ a Riemannian submersion of Lie groupoids. To see this, note that by construction the map $s:\tilde G\to\tilde M$ preserves the horizontal distributions, or in other words, the horizontal lift yield a commutative square:
$$\xymatrix@R=15pt{
\ar[r]^{\d s} \ar[d];[dr]^{\d s} \ar[d];[] \ar[dr];[r]
T_{\tilde g}\tilde G & T_{\tilde x}\tilde M \\ T_gG & T_xM.
}$$
In particular $\ker\d_gs$ lifts to $\ker\d_{\tilde g}s$, from where the horizontal lift $T_{a}G^{[3]}\to T_{\tilde a}\tilde G^{[3]}$ preserves the canonical orthogonal decomposition of the fiber product metrics, and is therefore an isometry.

\medskip

\emph{Step 3:} Averaging of the metrics on $\tilde G^{[3]}$ and $G^{[3]}$.
\smallskip

This last step consists in averaging the two metrics on $\tilde G^{[3]}$ and $G^{[3]}$ with respect to the right actions of $\tilde G$ and $G$, respectively. If one uses for $\tilde G$ a fiber averaging data, by Proposition \ref{prop:fiberd-averaging}, we conclude that the resulting metrics descend to 2-metrics $\tilde\eta^{(2)}$ and $\eta^{(2)}$ on the original groupoids $\tilde G$ and $G$ and make the following diagram a commutative diagram of Riemannian submersions:
\[
\xymatrix@R=15pt{
\tilde G^{[3]}\ar[r]\ar[d] &\tilde  G^{(2)}\ar[d] \\
G^{[3]}\ar[r] &  G^{(2)} 
}
\]
Hence,  $\tilde\eta^{(2)}$ and $\eta^{(2)}$ fulfill all the desired properties.
\end{proof}


\section{Linearization of fibrations}
\label{sec:linear:fibrations}


In this section we show our main theorem, which asserts that any fibration between proper groupoids can be linearized. We achieve this by applying the results in \cite{dhf} on linearization of groupoid by exponential maps, combined with our construction of Riemannian metrics on split fibrations (Theorem \ref{thm:Riemannian:submersions}), and relating any fibration with a split one by means of Lemma \ref{lemma:associated-fibration}.  
This is the cornerstone from which we will later derive as applications our results on rigidity of Lie groupoids, metrics over stacks and the stacky Ehresmann's Theorem.

\subsection{Linearization of groupoids}

Let $G\toto M$ be a Lie groupoid and let $S\subset M$ be a {\bf saturated} submanifold, so the restriction $G_S\toto S$ is a Lie subgroupoid. The {\bf linear local model} of $G$ around $S$ is the groupoid-theoretic normal bundle $\nu(G_S)\toto \nu(S)$, whose objects and arrows are given by $\nu(S)=T_SM/TS$ and $\nu(G_S)=T_{G_S}G/TG_S$, and structure maps are induced by differentiating those of $G$. A {\bf linearization} of $G$ around $S$ is a pair of open subgroupoids $G_S\subset U\subset G$ and $G_S\subset V\subset \nu(G_S)$, together with a groupoid isomorphism  $\alpha: U \xto\cong  V$.

Recall that a subgroupoid $U\subset G$ is {\bf full} if it contains every arrow of $G$ between its objects.
We say that:
\begin{itemize}
\item $G$ is {\bf weakly linearizable} around $S$ if there exists some linearization;
\item $G$ is {\bf linearizable} around $S$ if one can take $U,V$ to be full subgroupoids;
\item $G$ is {\bf strictly linearizable} if, in addition, $\U0$ and $\V0$ can be taken to be saturated.
\end{itemize}
These different notions may agree if $G$ is ``nice enough'': it is proved in \cite{dhf} that any open subgroupoid $U$ can be shrinked to a full one when $G$ is proper, and to a full saturated one when $G$ is source-proper. The various results on linearization of groupoids can be derived from the following fundamental theorem:

\begin{theorem}[\cite{dhf}]
\label{thm:linearization:grpd}
Given $(G\toto M,\eta^{(2)})$ a Hausdorff Riemannian groupoid and given $S\subset M$ a saturated embedded submanifold, the exponential maps arising from $\eta$ define a (weak) linearization of $G$ around $S$.
\end{theorem}



This follows by extending the usual method to build tubular neighborhoods on manifolds using metrics. Given $(M,\eta)$ a Riemannian manifold and $S\subset M$ an embedded submanifold, we say that an open $S\subset U\subset TS^\bot$ is {\bf admissible} if $\exp|_U$ is an open embedding. 
The existence of admissible opens for manifolds is a consequence of the following well-known result: {\it If $S\subset M$ is an embedded submanifold and $f:M\to N$ is such that $f|_S$ is an embedding and $\d_xf$ is invertible for all $x\in S$, then $f|_U$ is an open embedding on some open $S\subset U$.} The key point here is to show global injectivity.

In order to extend this to groupoids, given a Riemannian groupoid $G\toto M$ and an embedded saturated submanifold $S\subset M$, we will say that an open subgroupoid $G_S\subset U\subset \nu(G_S)$ is {\bf admissible} if (i) its objects $\U0$ form an admissible open, and (ii) the pairs of composable arrows $\U2=\U1\times_{\U0}\U1$ are within the domain of the exponential map of $\eta^{(2)}$. We then have:

\begin{lemma}\label{lemma:admissible}
If $G\toto M$ is a proper groupoid, $S\subset M$ is saturated embedded and $\phi:G\to H$ is a Lie groupoid map such that $\phi|_{G_S}$ is an embedding and $\d_g\phi$ is invertible for all $g\in G_S$ then $\phi|_{G_U}$ is an open embedding for some open $S\subset U$.
\end{lemma}

\begin{proof}
By the manifold version of the lemma applied at the level of the arrows, we know of the existence of an open $G_S\subset V\subset G$ over which $\phi^{(1)}$ is an open embedding. By \cite[Lemma 5.3]{dhf} there exists an open $S\subset U\subset M$ such that the restriction groupoid $G_U$ is included in $V$ (even though the statement of that lemma requires $V$ to be a subgroupoid, this is not used in the proof). The result follows.
\end{proof}

The following simple example shows that the properness assumption in the previous lemma is crucial.

\begin{example}\label{ex:admissible}\
Let $G$ be the constant group bundle over $\R$ with fiber $(\R,+)$, and $K\subset G$ the subgroupoid with fiber $K_t=(1/t)\Z$. The quotient $G/K$ is a group bundle with fiber $\Ss^1$ if $t\neq0$, and fiber $\R$ at $t=0$. The projection $G\to G/K$ is an embedding at $S=G_0$, but not in any open $S\subset U$. 
\end{example}

This way proper groupoids admit 2-metrics and full admissible open neighborhoods around saturated embedded submanifolds, so they can be linearized by exponential maps. More generally, the existence of admissible open subgroupoids holds for arbitrary {\em Hausdorff} Riemannian groupoids (see Propostion 5.9 and Theorem 5.11 in \cite{dhf}, where the assumption of Hausdorff is missing).





\subsection{Linearization of fibrations}


Let $\phi:\tilde G\to G$ be a submersion and $S\subset M$ a saturated submanifold. Then $\tilde S=\phi^{-1}(S)\subset \tilde M$ is a saturated submanifold as well, and we have an induced map $\overline{\d\phi}:\nu(\tilde G_{\tilde S})\to \nu(G_{S})$ between the corresponding linear local models:
$$
\xymatrix@R=20pt{ \nu(\tilde G_{\tilde S})\ar@<0.20pc>[r] \ar@<-0.20pc>[r] \ar[d]_{\overline{\d\phi^{(1)}}} &
\nu(\tilde S))\ar[d]^{\overline{\d\phi^{(0)}}}
\\
\nu(G_S)\ar@<0.20pc>[r] \ar@<-0.20pc>[r] & \nu(S)}$$

\begin{definition}
\label{defn:linearozation:submersion}
We say that $\phi$ is {\bf linearizable around $S$} if there is a linearization $\tilde \alpha$ of $\tilde G$ around $\tilde S=\phi^{-1}(S)$ and a linearization $\alpha$ of $G$ around $S$ forming a commutative square:
$$\xymatrix{
\nu(\tilde G_{\tilde S}) \supset
\tilde U \ar[d]_{\overline{\d\phi}} \ar[r]^--{\tilde\alpha}
& \tilde V \subset \tilde G \ar[d]^\phi \\
\nu(G_S) \supset U \ar[r]^--{\alpha} & V \subset G}$$
\end{definition}


As one could guess, the existence of compatible 2-metrics on $\tilde G$ and $G$ provide a linearization by exponential maps. 

\begin{proposition}
\label{prop:linearization:submersions}
Let $\phi:(\tilde G,\tilde \eta)\to (G,\eta)$ be a Riemannian submersion between Riemannian groupoids, let $S\subset M$ be embedded saturated, and let $\tilde S=\phi^{-1}(S)$. Then the exponential maps of $\tilde \eta$ and $\eta$ define a linearization of $\phi$ around $S$.
\end{proposition}

\begin{proof}
We know that the exponential maps restricted to an admissible open subgroupoid defines a linearization. Hence, it is enough to construct admissible open subgroupoids $\tilde U$ and $U$ such that $\phi(\tilde U)\subset U$, for then we obtain a linearization diagram:
$$\xymatrix@C=50pt{
\nu(\tilde G_{\tilde S}) \supset
\tilde U \ar[d]_{\overline{\d\phi}} \ar[r]^{\exp_{\tilde \eta}}
& \tilde V \subset \tilde G \ar[d]^\phi \\
\nu(G_S) \supset U \ar[r]^{\exp_\eta} & V \subset G}$$
We denote by $\E$ the domain of the exponential map. Let us fix $S\subset \U0\subset \nu(S)$ an admissible open, define $\tilde U^{(0)}=\E_{\tilde M}\cap\phi^{-1}(\U0)$, and now consider
\begin{align*}
\U1&=\d s^{-1}(\U0)\cap \d t^{-1}(\U0)\cap \E_G\cap \nu(G_S)\\
{{\tilde U}^{(1)}}&=\d s^{-1}(\tilde U^{(0)})\cap \d t^{-1}(\tilde U^{(0)})\cap \E_{\tilde G}\cap \nu({\tilde G}_{\tilde S})\cap \phi^{-1}(\U1)
\end{align*}
As in the proof of Theorem \ref{thm:linearization:grpd}, $U$ is an admissible open subgroupoid. Then $\tilde U$ is the intersection of an admissible open subgroupoid with the preimage of an open subgroupoid, and therefore, it is an admissible open subgroupoid too.
\end{proof}

%


According to Theorem \ref{thm:Riemannian:submersions} for a fibration $\phi:\tilde G\to G$ between proper Lie groupoids to be Riemannian it is enough to be split. We do not know if this condition is necessary. Nevertheless, the splitting condition is fulfilled in important examples, such as those we will use in the subsequent sections to prove rigidity of compact groupoids, and the Morita invariance of metrics. Moreover, one has the following stronger linearization result.

\begin{theorem}
\label{thm:linear:fib}
Any fibration $\phi:\tilde G\to G$ between proper Lie groupoids is linearizable around any saturated embedded submanifold $S\subset M$.
\end{theorem}

\begin{proof} We split the proof into several cases.
\medskip

\emph{Case 1:} $\phi$ admits a unital flat cleavage.
\smallskip

In this case $\phi$ is a split fibration so, by Theorem \ref{thm:Riemannian:submersions}, it can be made Riemannian. By  Proposition \ref{prop:linearization:submersions}, we can linearize $\phi$ using the exponential maps.
Note that, since $\tilde G$ and $G$ are proper, we can take the open neighborhoods to be full.

\medskip

\emph{Case 2:} $\phi$ admits a unital cleavage.
\smallskip

We use the canonical factorization $\phi=\phi'\iota$, explained in Section \ref{section:split-fibration}. Since \smash{$\iota:\tilde G\to \tilde G'$} is an equivalence and $\tilde G$ is proper, $\tilde G'$ is proper as well. Endow the split fibration $\phi':\tilde G'\to G$ with a Riemannian submersion structure. Identify $\tilde G $ with a subgroupoid of $\tilde G'$ via $\iota$, and consider the retraction $r:\tilde G'\to \tilde G$ given by the cleavage (see Lemma \ref{lemma:associated-fibration}). If we set $\tilde S'=\phi'^{-1}(S)$, by Proposition \ref{prop:linearization:submersions}, there are admissible open neighborhoods giving a linearization of $\phi'$ around $S$. Since the ambient groupoids are proper, we can take these admissible opens to be full. By restricting the VB-groupoid $\nu(\tilde G'_{\tilde S'})$ to $\tilde G_{\tilde S}$ we get a groupoid diagram as follows:
$$\xymatrix{
 & \nu(\tilde G'_{\tilde S'})|_{\tilde G_{\tilde S}} \supset {\tilde U'}
 \ar[rr]^{\exp_{\tilde\eta'}} 
 \ar[dl]_{\overline{\d r}}
 \ar[dd]_(0.35){\overline{\d\phi'}} & & \tilde V'\subset \tilde G'  \ar[dd]^(0.35){\phi'}
 \ar[dl]_{r} \\
 \nu(\tilde G_{\tilde S}) \ar[dr]^{\overline{\d\phi}} \ar@{-->}[rr]& &   \tilde G \ar[dr]^{\phi}& \\
 & \nu(G_S)\supset U \ar[rr]^{\exp_{\eta}} &  & V\subset G 
}$$

Here $\tilde U'=\tilde U''|_{\tilde G_{\tilde S}}$, where $\tilde U''\subset\nu(\tilde G'_{\tilde S'})$ is an admissible open groupoid, $\tilde V'=\exp_{\tilde\eta'}(\tilde U')$ is an embedded subgroupoid of $\tilde G'$ containing $\tilde G_{\tilde S}$,
$\exp_{\tilde\eta'}:\tilde U'\to \tilde V'$ is a groupoid isomorphism and $\overline \d r:\nu(\tilde G'_{\tilde S'})|_{\tilde G_{\tilde S}}\to\nu(\tilde G_{\tilde S})$ is an isomorphism. The square and the triangles commute.  In general, $\tilde V'$ is neither open nor included in $\tilde G$.

Consider the composition
\[ \alpha=r\circ\exp_{\tilde\eta'}\circ (\overline{\d r})^{-1}:\overline{\d r}(\tilde U')\to \tilde G. \] 
Note that $\alpha|_{\tilde G_{\tilde S}}$ is an open embedding and $\d_g\alpha$ is invertible for $g\in\tilde G_{\tilde S}$. Moreover, $\overline{\d r}(\tilde U')$ is a full subgroupoid of a proper groupoid, hence it is proper. Then, by Lemma \ref{lemma:admissible}, we conclude that $\alpha$ defines an open embedding of groupoids on some full neighborhood $\tilde U$ of $\tilde G_{\tilde S}$, yielding a linearization of $\phi$.

\medskip

\emph{Case 3:} $\phi$ is any fibration.

Recall that a cleavage for $\phi$ is a global section $\sigma$ for $\hat\phi:\tilde G\to \tilde M'=G\times_M\tilde M$, and that the cleavage is unital if it extends the section $\tilde u:\tilde M\to \tilde G$, where $\tilde M$ is seen as a submanifold of $\tilde M'$ via $\tilde x\mapsto (\id_{\phi(\tilde x)},\tilde x)$. Even when no such cleavage exists, we can still construct a {\em local cleavage}. By this we mean an horizontal lift for arrows $g$ that are closed to an identity, in the form of a section $\sigma:W\to \tilde G$ extending $\tilde u$, defined over a tubular neighborhood $\tilde M\subset W\subset \tilde M'$. One way to do this is by linearizing $\hat\phi$ around $\tilde M$. Then the section $\sigma$ gives rise to a Lie groupoid retraction $r:\tilde G'_W\to \tilde G$, with the same formulas as in Lemma \ref{lemma:associated-fibration}, now defined only over the restriction of $\tilde G'$ to $W$. The proof now follows as in the previous case, by eventually shrinking the opens $\tilde U',\tilde V'$ so as to insure $\tilde V'\subset \tilde G'_W$.
\end{proof}


The classical Ehresmann's Theorem states that proper submersions between manifolds are locally trivial. This can be seen as a combination of two facts: (i) submersions are linearizable, and (ii) any open around the fiber of a proper map contains a tube, i.e., a saturated open neighborhood of the fiber. 
In light of Theorem \ref{thm:linear:fib} and the tube principle \ref{prop:tube}, we have the following groupoid version:

\begin{corollary}[Groupoid Ehresmann]\label{thm:groupoid-ehresmann}
Let $\phi:\tilde G\to G$ be a proper fibration between proper Lie groupoids, and let $S\subset M$ be a saturated submanifold. Then there exist open full subgroupoids $G_S\subset U\subset \nu(G_S)$ and $G_S\subset V\subset G$, and linearization maps:
$$\xymatrix@R=20pt{
\nu(\tilde G_{\tilde S}) \supset(\overline{\d\phi})^{-1}(U)\ar[d]_{\overline{\d\phi}} \ar[r]^---{\tilde\alpha}
& \phi^{-1}(V) \subset \tilde G \ar[d]^\phi \\
\nu(G_S) \supset U \ar[r]^--{\alpha} & V\subset G}$$
\end{corollary}




\section{Rigidity of compact groupoids}
\label{section:rigidity}


The theory developed so far, comprising some basic facts on fibrations and metrics within the groupoid framework, is actually quite powerful. 
We illustrate this by deriving from it some results on the structural stability of Lie groupoids. These results, that generalize classic deformation theorems in the geometry of actions, fibrations and foliations, were obtained recently in \cite{cms} by different, less direct, methods.



Recall that we have defined a family of Lie groupoids parametrized by $M$ to be a fibration $(\tilde G\toto \tilde M)\to (M\toto M)$ over the unit groupoid. Given $x\in M$, we are interested in the behavior of the fiber $\tilde G_x\toto \tilde M_x$ when $x$ varies. 

\begin{definition}
Let $G\toto M$ be a Lie groupoid and $I$ a manifold with base point $0\in I$. A {\bf deformation} of $G$ parametrized by $I$ consists of a family of Lie groupoids
$$\phi:(\tilde G\toto \tilde M)\to (I\toto I)$$
such that $\phi^{(1)},\phi^{(0)}$ are locally trivial and the central fiber $\tilde G_0$ is isomorphic to $G$. The deformation is called {\bf proper} if $\phi$ is a proper map (cf.~Definition \ref{def:proper-map}), and it is called {\bf trivial} if it is equivalent to the product family $G\times I\to I$. 
\end{definition}



The most relevant case is when $I\subset \R$ is an open interval, in which case we speak of a {\bf 1-parameter deformation}. Of course, we can also consider {\bf $k$-parameter deformations}, where $I\subset\R^k$ is some open set, and in general we can allow $I$ to be any manifold. Note that proper deformations of $G$ only make sense when $G$ is compact. Also, in a proper deformation, the conditions that $\phi^{(1)}$ and $\phi^{(0)}$ be locally trivial become automatic.


In a deformation $\tilde G\to I$, the fact that $\phi^{(1)}$ and $\phi^{(0)}$ are assumed locally trivial, allows us to think that the manifolds $G$ and $M$ remain fixed while, when the parameter $\eps\in I$ varies, one deforms the structure maps $s_\eps,t_\eps,m_\eps,u_\eps,i_\eps$ of the groupoid $G$. More precisely, we obtain for each $\eps\in I$ a groupoid structure on the manifold $G=\tilde G_0$ by conjugating with some trivializations $\alpha^{(0)}:\tilde M\to M\times I$ and $\alpha^{(1)}:\tilde G\to G\times I$. For instance, the source $s_\eps$ will be given by
$$(s_\eps(g),\eps)=\alpha^{(0)}\circ\tilde s\circ(\alpha^{(1)})^{-1}(g,\eps) \qquad g\in G, \eps\in I.$$

\begin{examples}\label{deformations}\

\begin{enumerate}
\item (cf. \cite[Example 4.1.3]{dhf}) Let $G=]0,+\infty[\times\R$ and consider a family of groups $\phi:G\times \R\to\R$, $(x,y,\eps)\mapsto \eps$,
with multiplication: 
\[ (x_1,y_1)\star_\eps (x_2,y_2):=(x_1+x_2,y_1+(x_1)^\eps y_2). \]
This is a non-trivial 1-parameter deformation of the 2-dimensional abelian Lie group, since for $t\not=0$ the multiplication $\star_\eps$ is non-abelian.

\item Consider the family of group actions of $\R$ on $\mathbb{T}^2=\R^2/\Z^2$ given by:
\[ x\star_\eps (\theta_1,\theta_2)= (\theta_1+ x,\theta_2+\eps x). \]
This can be seen as a non-trivial 1-parameter deformation $\phi:G\times \R\to \R$, for clearly the topology of the orbits varies with $\eps$. 

\item Let $F_0$ be the foliation on $\Ss^1\times\R$ given by the projection on $\R$. The 1-forms $\epsilon\d r+\d t$ and $\d\epsilon$ define a foliation $F$ on $\Ss^1\times\R\times\R$ tangent to the fibers of the projection $(x, t,\epsilon)\mapsto\epsilon$. We can think of it as a one-parameter deformation of $F_0$, that is non-trivial: if $\epsilon\neq0$ the leaves of $F_\epsilon$ are diffeomorphic to $\R$. Writing $\R\to \Ss^1$, $x\mapsto e^{ix}$ for the universal cover, the monodromy groupoid $\mon(F)\toto \Ss^1\times\R\times\R$ identifies with a quotient of the groupoid arising from the submersion $\R\times\R\times\R\to\R\times\R$, $(x,t,\epsilon)\mapsto(t+\epsilon x,\epsilon)$. It is a non-trivial 1-parameter deformation of $\mon(F_0)\toto \Ss^1\times\R$.

\end{enumerate}
\end{examples}

Here we use our linearization results for fibrations to achieve a simple, geometric, proof that proper families of Lie groupoids are trivial. This rigidity result was obtained independently by M. Crainic, J. Mestre and I. Struchiner \cite{cms} by developing a deformation theory of Lie groupoids based in cohomological methods. 



\begin{theorem}
[{\cite[Thm 7.4]{cms}}]
\label{thm:rigidity}
Compact Lie groupoids are rigid: a proper deformation $\tilde G\to I$ of a compact Lie groupoid  $G\toto M$ is locally trivial.
\end{theorem}

\begin{proof}
By working locally we can assume $I=\R^k$. Let $G$ be a compact Lie groupoid and let $\phi:(\tilde G\toto \tilde M)\to(\R^k\toto\R^k)$ be a  deformation. Since $\phi$ is a proper fibration, it follows from Corollary \ref{thm:groupoid-ehresmann} that the family is linearizable over open sets $U,\tilde U=\phi^{-1}(U)$. Observe that the local linear model around the central fiber is just the trivial family $G\times\R^k\to\R^k$. The result follows.
\end{proof}



Theorem \ref{thm:rigidity} is a far-reaching generalization of several classic rigidity results in differential geometry, which one can deduce as immediate corollaries:
\begin{itemize}
\item the rigidity of Lie group structures on a compact manifold, implied by the deformation theory of Lie algebras (\cite{nr}); 
\item the rigidity of smooth actions $K\action M$ of a fixed compact group on a fixed compact manifold, obtained by Palais and Stewart (\cite{palais}).
\end{itemize}

The classical results on rigidity of compact fibrations, and more generally foliations, obtained by Epstein-Rosenberg and Hamilton (\cite{rosenberg1,hamilton}), also admit versions that can be derived from our framework. This is elaborated in \cite{dhf-note}.

It is natural to wonder if rigidity holds more generally in the context of proper Lie groupoids. In \cite[Rmk 7.3]{weinstein}, Weinstein raises the question of rigidity, first with fixed source and target, and then in the general case. For the general case, the answer is negative as can be shown by constructing an example of non-trivial deformation of proper group actions.

\begin{example}
Every smooth action $\rho:G\action\R^n$ with at least one fix point, say $0$, can be easily deformed into a linear action, by setting $\rho^\eps_g(x):=\frac 1 \eps \rho_g(\eps x)$. Thus, the construction of actions whose fixed locus is not diffeomorphic to a linear subspace is a source of examples of non-trivial deformations \cite{cms,pr}.
For a concrete simple example, let $\Z_2$ act on $\R$ by reflection in the origin and trivially on an exotic $\R^4_e$. Then the induced diagonal action $\rho:\Z_2\action \R\times \R^4_e\simeq\R^5$ is not isomorphic to a linear action, for its fixed point sets is just $\R^4_e$.
\end{example}


One can also use this example to produce s-proper foliation groupoids, with connected s-fibers, which are not rigid. These examples show that rigidity does not hold in general for proper groupoids or for source-proper groupoids, not even if the source map is locally trivial. However, using our methods, one can easily prove that deformations of proper groupoids with prescribed source and target are trivial, a result conjectured by A. Weinstein and also proved in \cite{cms} by less direct methods:

\begin{theorem}[{\cite[Theorem 7.3]{cms}}]
A $k$-parameter deformation $G\times\R^k\to \R^k$ of a proper groupoid $G$ that fixes the source and target is trivial.
\end{theorem}

\begin{proof}
By hypothesis, the source and target of the total groupoid $G\times\R^k\toto M\times\R^k$ can be written as product maps $s\times\id$ and $t\times\id$, though the multiplication, inversion and unit may vary with the parameter. Let $\eta^{(1)}$ be any 1-metric on $G\toto M$ (which exists, because $G$ is proper), and let $\eta_0$ be the euclidean metric on $\R^k$. The product metric $\eta^{(1)}\times\eta_0$ is a 1-metric on $\tilde G$ and the projection $G\times\R^k\to\R^k$ is Riemannian. Hence, we can use it as an input for the fibered gauge trick (see Theorem \ref{thm:Riemannian:submersions}), and obtain a 2-metric $\tilde\eta^{(2)}$ on $\tilde G$ that induces the metric $\eta^{(0)}\times\eta_0$ on the units (cf.~Lemma \ref{lemma:trick-fix}). It follows that $\tilde \exp^{(0)}:\nu(M)\to M\times\R^k$ is globally defined and a diffeomorphism. Since $G\to M\times M$ is proper, the same holds true for $\tilde\exp^{(1)}$. Hence, we have a groupoid isomorphism $\tilde\exp:\nu(G)\to G\times\R^k$. Since $\nu(G)\simeq G\times\R^k$ is just the trivial fibration, the result follows.
\end{proof}


\section{The stack perspective}
\label{sec:stack}

\def\C{{\mathcal C}} 
\def\op{^\circ}

A stack can be thought of as a generalization of the notion of space, which allows at the same time for internal symmetries and singular behavior, of the type one often encounters in the study of moduli spaces of geometric or algebraic structures. We refer for details to the original monograph by Giraud \cite{giraud}, or the more recent survey \cite{vistoli}. The smooth version, discussed for instance in \cite{bx,lerman,metzler}, admits an alternative formulation, avoiding the classic paraphernalia, under which  differentiable stacks are Lie groupoids modulo Morita equivalences. We adopt here this approach and we express our linearization results from the previous sections in the language of stacks. Before we can formulate our results, will we need to develop some basic notions related to maps between stacks, such as immersions and submersions, which do not seem to be available in the literature.

\subsection{Morita maps and differentiable stacks}
Let us recall that a Lie groupoid map $\phi:\tilde G\to G$ is {\bf Morita} if it is 
 {\bf fully faithful} and  {\bf essentially surjective}. The first condition means that it induces a good fiber product of manifolds (\cite[2.2]{dh})
$$\xymatrix{\ar@{}[dr]|{\rm fp}
\tilde G \ar[r]^\phi \ar[d]_{(\tilde s\times \tilde t)}
& G \ar[d]^{(s,t)}\\
\tilde M\times \tilde M \ar[r]^{\phi\times\phi} & M\times M,}$$
while the second condition means that the following map is a surjective submersion:
$$G\times_{M}\tilde M\to M \qquad (y\xfrom{g} \phi(x),x)\mapsto y.$$

A Lie groupoid $G\toto M$ defines an {\bf orbit stack} $[M/G]$. Two Lie groupoids $G,H$ are {\bf Morita equivalent}, or equivalently, yield isomorphic orbit stacks, if there is a third Lie groupoid $\tilde G$ and Morita maps $G\from \tilde G\to H$. By a {\bf differentiable stack} we will mean the orbit stack $[M/G]$ of some groupoid.


In \cite[Theorem 4.3.1]{dh} the notion of Morita map $\phi:\tilde G\to G$ was reformulated in terms of the map induced between the orbit spaces $\bar\phi:\tilde M/\tilde G\to M/G$, the morphisms between the isotropy groups $\phi_x:\tilde G_x\to G_{\phi(x)}$, and the morphisms between the normal vector spaces to the orbits $\overline{\d_x\phi}:\nu_x(O)\to \nu_{\phi(x)}(O)$. We need the following slight improvement, which gives independent characterizations of fully faithful and essentially surjective maps. 

\begin{proposition}\label{prop:characterization}
Let $\phi:\tilde G\to G$ be a Lie groupoid map. Then:
\begin{enumerate}[(i)]
 \item $\phi$ is fully faithful if and only if $\bar\phi$ is injective, $\phi_x$ is an isomorphism for all $x\in \tilde M$, and $\overline{\d_x\phi}$ is a monomorphism for all $x\in \tilde M$.
 \item $\phi$ is essentially surjective if and only if $\bar\phi$ is surjective and $\overline{\d_x\phi}$ is an epimorphism for all $x\in \tilde M$; in such a case, the map $\bar\phi$ is open.
\end{enumerate}
Thus, $\phi$ is Morita if and only if $\bar\phi$ is a homeomorphism, and $\phi_x$ and $\overline{\d_x\phi}$ are isomorphisms for all $x\in \tilde M$.
\end{proposition}

\begin{proof}
In \cite[Thm 4.3.1]{dh} it is proved the second statement and half of the first one. It only remains to show that if $\bar\phi$ is injective and $\phi_x$ is a monomorphism for every $x$, then $\phi$ is fully faithful. Under these assumptions, it is easy to see that the anchor maps define a set-theoretic fiber product. To see that it is in fact a good fiber product of manifolds we need to show that the subset
$$S=\{(t(g),s(g),\phi(g)):g\in \tilde G\}\subset \tilde M\times \tilde M\times G$$
is an embedded submanifold with the expected tangent space. Note that $S$ can be written as the intersection of the two embedded submanifolds
$$S_1=\{(y,x,g'): s(g')=\phi(x)\} \qquad S_2=\{(y,x,g'): t(g')=\phi(y)\}$$
and, from the assumption on the tangent spaces, the intersection is {\bf clean}, namely 
$T_xS=T_xS_1\cap T_xS_2$ for all $x\in S$. Item (i) now follows from the general fact that clean intersection of embedded submanifolds is embedded, that we leave as exercise. 
\end{proof}

The previous proposition gives some geometric intuition into the notion of orbit stack of $G\toto M$: it is an enhanced version of the orbit space $M/G$, endowed with smooth information encoded in the normal representations $G_x\action \nu_x(O)$.



A map $\phi:\tilde G\to G$ is a {\bf Morita fibration} if it is Morita and it is a fibration. There is some redundancy among these axioms: $\phi$ is a Morita fibration if and only if it is fully faithful and a surjective submersion on objects, or equivalently, if and only if $\phi$ is a fibration whose kernel is the submersion groupoid $\tilde M\times_M\tilde M\toto \tilde M$.

As a converse for the last statement, if $\tilde G\toto \tilde M$ is a Lie groupoid and $K\toto \tilde M$ is a subgroupoid that is proper and free (trivial isotropy), we can define a {\bf quotient groupoid} $G\toto M$ by setting $M=\tilde M/K$ and $G=\tilde G/K\times K$, where $K\times K\action \tilde G$ is the action by left and right multiplication. These are well-defined manifolds (see eg. \cite[2.3]{dh}), and the structure maps descend to yield a Lie groupoid $G\toto M$, for which the projection $\tilde G\to G$ is a Morita fibration.
This way we get the following:

\begin{proposition}
\label{prop:quotient}
There is a 1-1 correspondence between Morita fibrations $\tilde G\to G$ and free proper wide subgroupoids $K\to \tilde G$.
\end{proposition}


\subsection{Stacky immersions, submersions and embeddings}
\label{sec:mapsstacks}



Recall the discussion of homotopy fiber product in Section \ref{section:split-fibration} (see also \cite{dh, mm}). 
Given $\phi_1:G_1\to H$ and $\phi_2:G_2\to H$ Lie groupoid maps, their  homotopy fiber product $G_1\tilde\times_H G_2$ fits into a square that commutes up to isomorphism of Lie groupoid maps, and that is universal for that property.
$$\xymatrix{ \ar@{}[dr]|{\rm hfp}
G_1\tilde\times_H G_2 \ar[d]_{\tilde\phi_2} \ar[r]^{\tilde\phi_1} \ar@{}[dr]|(.7){} & G_2 \ar[d]^{\phi_2}\\ G_1 \ar[r]_{\phi_1}  & H}$$
We refer to $\tilde\phi_1$ and $\tilde\phi_2$ as the {\bf homotopy base change} of $\phi_1$ and $\phi_2$, respectively. The homotopy fiber product may not exist for general $\phi_1$ and $\phi_2$. The following is a sharp formulation of a well-known result (see e.g. \cite[4.4]{dh}). It readily implies that Morita maps are stable under homotopy base change.

\begin{proposition}
If $\phi_1$ is essentially surjective then the homotopy fiber product exists and its homotopy base change map $\tilde\phi_1$ is a split fibration.
Moreover, $\phi_1$ is fully faithful if and only if $\tilde\phi_1$ is so.
\end{proposition}

A {\bf map of stacks} $\psi/\phi:[\tilde M/\tilde G]\to [M/G]$ is given by a fraction of Lie groupoid maps,
\[ \xymatrix{(\tilde G\toto \tilde M)& (H\toto N) \ar[l]_{\phi}^{\sim}\ar[r]^{\psi}& (G\toto M)} \] 
 where the first leg $\phi$ is a Morita map. We identify two fractions $\psi/\phi$ and $\psi'/\phi'$ if there are  Morita maps $\alpha,\alpha'$ and isomorphisms of maps $\phi\alpha\cong\phi'\alpha'$, $\psi\alpha\cong\psi'\alpha'$. Maps of stacks can be composed using homotopy fiber products.
It is an instructive exercise to check that composition is well-defined, associative, and that a Morita equivalence is the same as an invertible map.

Immersion and submersion of stacks arise naturally, once we think of the orbit stack $[M/G]$ as a ``smooth orbit space", with the normal representation $G_x\action \nu_x(O)$ playing the role of the tangent space to $[M/G]$ at the point $[x]=O$. 
We say that a stacky map $\psi/\phi:[\tilde M/\tilde G]\to [M/G]$ is a {\bf surjective submersion} if $\psi$ is essentially surjective, and a {\bf  injective immersion} if $\psi$ is fully faithful.
It follows from Proposition \ref{prop:characterization} that surjective submersions and injective immersions are well-defined, and that they extend the usual notions for manifolds. 
Using the canonical factorization $\phi=\phi'\iota$ from Section \ref{section:split-fibration}, one can find nice representatives of such maps, showing that our definitions agree with those in \cite[Def 2.9]{bnz}.






In order to define embedding of stacks we need first to define what the image  is. Given $\phi:G\to H$ any Lie groupoid map, its {\bf essential image} consists of the objects and arrows that are in the image up to isomorphism, or more precisely, is the image of the associated map $\phi':G'\to H$ with respect to the factorization $\phi=\phi'\iota$.  

\begin{proposition}\label{prop:image}
The essential image of a fully faithful map $\phi:G\to H$ is a full saturated Lie subgroupoid $i:I\to H$.
\end{proposition}

\begin{proof}
In the canonical factorization $\phi=\phi'\iota$ the map $\iota$ is a categorical equivalence, hence Morita. It follows from Proposition \ref{prop:characterization} that $\phi'$ is fully faithful as well.
The kernel $K'$ of $\phi'$ is a well-defined embedded subgroupoid of $G'$, for the  manifold $K'$ can be seen as the preimage of the units along $G'\to M'\times_{N}H$, and this is a surjective submersion (see Lemma \ref{lemma:associated-fibration}). Since $\phi'$ is also fully faithful, the kernel $K'$ is free and proper. By Proposition \ref{prop:quotient}, we conclude that the quotient $G'/K'$ is a well-defined Lie groupoid, and that the quotient map is Morita. Then the induced map $G'/K'\to H$ is fully faithful and, by Proposition \ref{prop:characterization}, it is an injective immersion. Its image is clearly full and saturated.
\end{proof}

We can now define the {\bf image} of a stacky injective immersion $\psi/\phi:[\tilde M/\tilde G]\dashto [M/G]$ as the orbit stack of the essential image of $\psi$.
We say that a stacky map $\psi/\phi:[\tilde M/\tilde G]\dashto [M/G]$ is an {\bf embedding} if it is an injective immersion and the essential image of 
$\psi$ is an embedded subgroupoid. This extends the usual notion for manifolds. Our next lemma shows that this is a good definition, for it does not depend on the groupoids presenting the stacks.

\begin{lemma}
\label{lem:Morita:full:pullback}
Let $\phi:\tilde G\to G$ be a Lie groupoid fibration.
\begin{enumerate}[(i)]
 \item The pull-back $\phi^*(G_S)$ of a full (embedded) Lie subgroupoid $G_S$ is a full (embedded) Lie subgroupoid.
 \item If $\phi$ is Morita, the pull-back $G_S\mapsto \phi^*(G_S)$ defines a 1-1 correspondence between full (embedded) Lie subgroupoids of $G$ and $\tilde G$.
\end{enumerate}
\end{lemma}

\begin{proof}
If $\phi:\tilde G\to G$ is a fibration and $S\subset M$ is saturated or embedded then the preimage $\phi^{-1}(S)$ is clearly saturated and or  embedded. The first statement follows easily. Regarding the second statement, if $\phi:\tilde G\to G$ is a Morita fibration and $\tilde S\subset \tilde M$ is saturated, then it is also saturated with respect to the action of the kernel $K$, for the fibers are included in the orbit. Since the action $K\action \tilde S$ is free and proper, we can construct the orbit manifold $S=\tilde S/K\simeq\phi(S)$, that comes equipped with a canonical map $S\to M$. Note that $(\tilde S\to S)\to (\tilde M\to M)$ is a map of principal $K$-bundles, thus $\tilde S$ is (embedded) submanifold if and only if $S$ is so.
\end{proof}


%
%


A stacky map between the orbit stacks $\psi/\phi:[\tilde M/\tilde G]\to [M/G]$ induces a continuous map between the orbit spaces $\overline{\psi/\phi}:\tilde M/\tilde G\to M/G$ (cf.~Proposition \ref{prop:characterization}).
In the case of proper groupoids, whose orbit stacks are called {\bf separated stacks}, we have the following characterization of embeddings:

\begin{proposition}
A map $\psi/\phi$ between separated stacks is an embedding if and only if it is an injective immersion and $\overline{\psi/\phi}$ is a topological embedding.
\end{proposition}

\begin{proof}
It is easy to see that every stacky embedding satisfies this properties. So let us show the converse. We will make a number of reductions to simplify the proof. To start with, by Proposition \ref{prop:image}, we can assume that $\phi=\id$ and that $\psi:(G_S\toto S)\to (G\toto M)$ is the inclusion of a full saturated subgroupoid that induces a topological embedding in the orbit spaces. 

Given $x\in S$, by the linearization theorem, the restriction of $G$ to some saturated open $S\subset U\subset M$ is Morita equivalent to $K\ltimes V\toto V$, the action groupoid of a linear representation of a compact group. Hence, by Lemma \ref{lem:Morita:full:pullback}, we can assume that $G$ is of this type, so we set $G=K\times V$, $M=V$ and $x=0$.

We claim that the action $K\action V$ restricts to a smooth action $K\action S$, so that $G_S$ becomes the action groupoid $K\times S$. A priori, we have two different good fiber product of manifolds, with a map $\alpha:G_S\to K\times S$ relating them. They are depicted in the front and the back of the following cube:
{\small
$$\begin{matrix}
\xymatrix@R=8pt@C=8pt{
G_S  \ar[rr] \ar[dd]_{(t,s)} \ar[rd]_{\alpha}&  &K\times V \ar'[d][dd]^--{(\rho,\pi)} \ar@{=}[dr]& \\
 &K\times S \ar[dd]_--------{\pi} \ar[rr]&  & K\times V\ar[dd]^\pi\\
S\times S \ar'[r][rr] \ar[dr]_{\pi}&  &V\times V \ar[dr]_{\pi}& \\
& S \ar[rr]& & V}
\end{matrix}$$
}
Since $S$ is saturated, $\alpha$ is bijective. Since $G_S\to K\times S\to K\times V$ is an immersion, $\alpha$ is also an immersion, and we conclude that $\alpha$ is a diffeomorphism. The restricted action $K\times S\to S$ can therefore be written as $t\circ\alpha^{-1}$, and so it is smooth.

Finally, we show that $S$ is embedded in $V$: given $x_n,x\in S$ such that $\lim x_n=x$ in $V$, we claim that convergence also holds in $S$. Consider the diagram:
$$
\xymatrix@R=15pt{ 
S \ar[d] \ar[r] & V \ar[d]\\ S/K \ar[r]  & V/K}
$$
The sequence $\bar x_n=x_n K$ has limit $\bar x=x K$ in $V/K$ and also in $S/K$, since the later is embedded. Now, the vertical maps are proper, being the quotient maps of compact group actions. Hence, there is a convergent subsequence $\{x_{n_k}\}$ in $S$, whose limit must be $x$. We conclude that $\lim x_n=x$ in $S$ and the result follows.
\end{proof}

It should be noted that if one drops the separated assumption then the previous characterization of an embedding does not hold: there are examples of full saturated subgroupoids that are not embedded but induce embeddings on the orbit space.

%

\subsection{Metrics on differentiable stacks}

\label{sec:stack2}

We are now ready to present our first main result on the geometry of stacks: a Morita invariance of groupoid metrics, which leads to a notion of metrics on differentiable stacks.
From the stack perspective a Lie groupoid is the same as a submersion onto the orbit stack:
$$G\toto M \qquad \rightleftharpoons \qquad M\to [M/G].$$

It is natural to expect that a 2-metric on a Lie groupoid $G\toto M$ induces a metric on the orbit stack $[M/G]$ making $M\to[M/G]$ a Riemannian submersion. In order to make this precise we need to understand how to relate 2-metrics on Morita equivalent groupoids. Note that different 2-metrics, even on the same groupoid, may lead to the same metric on the orbit stack, so we need to identify the resulting equivalence relation between 2-metrics.

By playing with homotopy fiber products, every Morita equivalence can be realized as a fraction of split Morita fibrations. Thus we are led to consider the pullback and pushforward of 2-metrics along Morita fibrations.
The pullback of 2-metrics along Morita fibrations is simpler:

\begin{proposition}
\label{prop:pullback}
If $\phi:\tilde G\to G$ is a Morita fibration and $\eta$ a 2-metric on $G$, then there exists a 2-metric $\tilde\eta$ on $\tilde G$ that makes the fibration Riemannian.
\end{proposition}

\begin{proof}
Let $V\subset T\tilde M$ be the vertical bundle of $\phi^{(0)}$ and let $E\subset T\tilde M$ be an Ehresmann connection for $\phi^{(0)}$, so $T\tilde M=E\oplus V$. 
Since $\phi$ is fully faithful the following are good fiber products of manifolds:
$$\xymatrix{
\ar@{}[dr]|{\rm fp}
\ar[];[d]_{\phi^{(1)}} \ar[];[r]^(.35){q} \ar[r];[dr]^{\phi^{(0)}\times\phi^{(0)}} \ar[d];[dr]_(.35){q'}
\tilde G & \tilde M \times \tilde M \\
G  &  M  \times M
}\qquad
\xymatrix{
\ar@{}[dr]|{\rm fp}
\ar[];[d]_{\phi^{(2)}} \ar[];[r]^(.35){q} \ar[r];[dr]^{\phi^{(0)}\times\phi^{(0)}\times\phi^{(0)}} \ar[d];[dr]_(.35){q'}
\tilde G^{(2)} & \tilde M \times \tilde M \times \tilde M \\
G^{(2)}  &  M \times M\times M.
}$$
It follows that we can identify the vertical bundles of $\phi^{(1)},\phi^{(2)}$ with the pullbacks of $V\times V$ and $V\times V\times V$, and that $E$ induces pullback connections $E^{(1)},E^{(2)}$ on $\phi^{(1)},\phi^{(2)}$. Hence, we can define a metric $\tilde\eta^{(2)}$ by lifting to $E^{(2)}$ the 2-metric $\eta^{(2)}$, using in $q^*(V\times V\times V)$ the product of some fixed metric $\eta^V$ in $V$, and declaring the decomposition
$T\tilde G^{(2)}=E^{(2)}\oplus q^*(V\times V\times V)$
to be orthogonal.
The resulting metric is $\phi^{(2)}$-fibred and a 2-metric on $\tilde G$, 
for the action $S_3\action \tilde G^{(2)}$ and the face maps $\tilde G^{(2)}\to\tilde G^{(1)}$ preserve the orthogonal decomposition, and the metrics on each factor are invariant under $S_3$ and transverse to the face maps.
\end{proof}


The pushforward of 2-metrics along Morita fibrations is harder. It is already clear in the case of manifolds (unit groupoids), that not every 2-metric can be pushed forward, or in other words, not every 2-metric on $\tilde G$ is $\phi$-transverse. In order to fix this problem, given a 2-metric on $\tilde G$ we slightly modify it so as to make it $\phi$-transverse:

\begin{proposition}
\label{prop:pushforward}
If $\phi:\tilde G\to G$ is a Morita fibration with kernel $K$ and $\tilde\eta$ is a 2-metric on $\tilde G$, then the cotangent average metric $\tilde\eta'$ on $K^3\action\tilde G^{(2)}$ with respect to a product averaged data is a 2-metric and descends to $G$, making $\phi$ 
Riemannian.
\end{proposition}

\begin{proof}
This requires many of the techniques we have developed before. Consider the three principal groupoid bundles, arising from Proposition \ref{prop:quotient}:
$$\theta^{(0)}:K\action \tilde M\to M \qquad 
\theta^{(1)}:K^2\action \tilde G\to G \qquad
\theta^{(2)}:K^3\action \tilde G^{(2)} \to G^{(2)}$$
Fix averaging data for $K\toto \tilde M$ and endow the products $K^2,K^3$ with the product averaging data. The various projections among these groupoids make these averaging data fibered, the sense of Section \ref{subsection:riemannian-submersion}. Now replace the metrics $\tilde\eta^{(2)},\tilde\eta^{(1)},\tilde\eta^{(0)}$ by their cotangent averages with respect to the above actions. It follows from Proposition \ref{prop:fiberd-averaging} that the resulting metrics make all the face maps Riemannian submersions. It is now easy to check that $\tilde\eta^{(2)}$ is a 2-metric that descends through $\phi$ so as to make it a Riemannian submersion.
\end{proof}


Motivated by the previous result, we say that 2-metrics $\eta_1,\eta_2$ on a Lie groupoid $G\toto M$ are {\bf equivalent} if for every $G$-orbit $O\subset M$ the metrics induced on the normal bundle $\nu(O)$ by the associated 0-metrics $\eta_1^{(0)},\eta_1^{(0)}$ coincide.
It turns out that our pullback and pushforward constructions are well-defined and mutually inverse modulo equivalence of metrics.

\begin{theorem}
\label{thm:Morita-metrics}
A Morita equivalence yields a 1-1 correspondence between equivalence classes of 2-metrics. In particular, if two Lie groupoids are Morita equivalent and one admits a 2-metric, then so does the other.
\end{theorem}

\begin{proof}
First, we observe that given a Morita fibration $\phi:(\tilde G,\tilde\eta_i)\to (G,\eta_i)$ that is Riemannian with respect to two pairs of 2-metrics ($i=1,2$), then $\eta_1$ is equivalent to $\eta_2$ if and only if $\tilde\eta_1$ is equivalent to $\tilde\eta_2$.  This follows because the fibers of $\phi$ are included in the orbits, namely the orbits of $\tilde G$ are of the form $\tilde O=\phi^{-1}(O)$, with $O$ an orbit of $G$. Therefore, the induced map on the normal vector bundles
$$ \overline{\d\phi}:\nu(\tilde O)\to \nu(O), $$
is a fiberwise isometry for each set of metrics. This proves that the pullback of metrics defined in \ref{prop:pullback} gives a well-defined injective map on classes of metrics. 

To see that it is surjective, it is enough to show that, in the pushforward construction of Proposition \ref{prop:pushforward}, the averaged metric $\tilde\eta'$, which descends to a metric on the base, is equivalent to the original metric $\tilde\eta$. For this, note that in $\nu({\tilde O})$ the metric $\tilde\eta$ is already $\tilde G$-invariant, and since the $K$-action preserves the $\tilde G$-orbit $\tilde O$, the metric $\tilde\eta$ over $\nu({\tilde O})$ is also $K$-invariant, and it remains the same after averaging.
\end{proof}

The previous theorem suggests a definition for Riemannian metrics over differentiable stacks.
We define a {\bf metric} on the orbit stack $[M/G]$ of a Lie groupoid $G\toto M$ as an equivalence class of a 2-metric $\eta^{(2)}$ on $G$.
This notion of metric generalizes the usual notions of metrics for manifolds and orbifolds, and allow us to perform Riemannian geometry on more general differentiable stacks.

\begin{examples}\label{isometries}\

\begin{enumerate}
 \item (Manifolds) For a unit groupoid $M\toto M$, a 2-metric is the same  as a metric on $M$, and distinct 2-metrics are always inequivalent. It follows that for every proper groupoid without isotropy equivalence classes of metrics are in 1:1 correspondence with metrics on the orbit manifold. Hence, our definition extends the usual definition of metric for manifolds.

 \item (Orbifolds) When $G\toto M$ is a proper effective \'etale groupoid the orbit stack $[M/G]$ is an (effective) orbifold (cf. \cite{mm}). A 2-metric on $G\toto M$ is determined by a $G$-invariant metric on $M$ (see \cite[Example 4.1]{dhf}). It follows that 2-metrics on $G$ are equivalent if and only if they induce the same orbifold metric on $[M/G]$ in the usual sense on $M/G$.

 \item (Lie groups) In a transitive Lie groupoid $G\toto M$ any two 2-metrics on $G$ are equivalent. Metrics on a Lie group, viewed as a stack, are trivial in the sense of our definition, which does not detect any relevant information on the isotropies, and only sees the transverse directions.
\end{enumerate}
\end{examples}


In the forthcoming paper \cite{dhdm} we develop the theory of these metrics over differentiable stacks, exploring the corresponding notion of geodesics, and establishing a stacky version of Hopf-Rinow theorem, among other results.



\subsection{Tubular neighborhoods and stacky Ehresmann}


A natural application of metrics on stacks is the construction of tubular neighborhoods of substacks. We first need to show that the local linear model (normal bundle) of a stack around an embedded stack is well-defined. This is a consequence of the following proposition.

\begin{proposition}
Let $\phi:\tilde G\to G$ be a Morita fibration, let $S\subset M$ be a saturated embedded submanifold, and let $\tilde S=\phi^{-1}(S)$. Then the induced map $\overline{\d\phi}:\nu(\tilde G_{\tilde S})\to \nu(G_S)$ is a Morita fibration.
\end{proposition}

\begin{proof}
The map $\overline{\d\phi}$ is a surjective submersion on objects, so we only need to show that it is also fully faithful. This follows from a standard argument on the commutative cube below: since the left, right and front face are good fiber products of manifolds, then so does the bottom face.
{\small
$$\begin{matrix}
\xymatrix@R=6pt@C=6pt{
\nu(\tilde G_{\tilde S}) \ar[rr] \ar[dd] \ar[rd]&  & \nu(\tilde S)\times \nu(\tilde S) \ar'[d][dd] \ar[dr]& \\
 & \tilde G_{\tilde S} \ar[dd] \ar[rr]&  & \tilde S\times \tilde S\ar[dd]\\
\nu(G_S) \ar'[r][rr] \ar[dr]&  & \nu(S)\times \nu(S) \ar[dr]& \\
& G_S \ar[rr]& & S\times S}
\end{matrix}$$
}
\end{proof}



We have seen in Section \ref{sec:mapsstacks} that an embedding between differentiable stacks can be modeled by the inclusion $G_S\to G$ of a full saturated embedded subgroupoid $G_S\toto S$ into $G\toto M$. We define the {\bf normal bundle} $\nu([S/G_S])$ of the corresponding stack embedding as the orbit stack of the groupoid $\nu(G_S)\toto \nu(S)$. The previous proposition shows that this is well-defined.
Notice that the orbit stack $[S/G_S]$ is also a substack of $[\nu(S)/\nu(G_S)]$ by means of the zero section. 

Having established a notion of normal bundle, we can easily make sense of linearization of a stack around a substack, i.e., tubular neighborhoods.
Let $[M/G]$ be the orbit stack of $G\toto M$ and let $[S/G_S]$ be an embedded substack modeled by a full saturated embedded subgroupoid $G_S\toto S$. A {\bf tubular neighborhood} of $[S/G_S]$ in $[M/G]$ is a stack isomorphism 
$[U/\nu(G_S)_U]\cong[V/G_V]$ between open embedded substacks of $[\nu(S)/\nu(G_S)]$ and $[M/G]$ that restricts to the identity on $[S/G_S]$.

A linearization of a groupoid $G\toto M$ around a saturated submanifold $S\subset M$ yields a tubular neighborhood of the substack $[S/G_S]$ inside $[M/G]$. A priori, the groupoid linearization is stronger than the existence of a stacky tubular neighborhood, for it involves groupoid isomorphism rather than Morita equivalences. However, these two notions are actually equivalent.

\begin{proposition}[Morita invariance of linearization]
Given $G\toto M$ a Lie groupoid and $S\subset M$ saturated embedded, 
the orbit stack $[S/G_S]$ admits a tubular neighborhood in $[M/G]$ if and only if $G$ is linearizable around $S$.
\end{proposition}

This type of result on Morita invariance of linearization goes back to the original paper \cite{weinstein}. The statement presented here generalizes (and is strongly inspired in) a result of \cite{cs}. 

For the proof, we follow the alternative approach to Morita equivalences by means of principal bibundles. Its correspondence with generalized maps is explained, e.g., in \cite{dh}. We will use the fact that a bibundle represents an actual Lie groupoid map if and only if it admits a global section, and it represents the identity map if and only if it is isomorphic to the trivial bundle \smash{$G\xfrom s M\xto t M$}.

\begin{proof}
Supposte there exists a stacky tubular neighborhood, namely a Morita equivalence between saturated neighborhoods $G_S\subset G_U\subset G$ and $G_S\subset \nu(G_S)_V\subset \nu(G_S)$, given by a bibundle $P$, whose restriction to $S$ is trivial:
$$\xymatrix{
\ar[rr];[r]_s \ar[rr];[rrr]^t \ar[drr];[dr]_{q_U} \ar[drr];[drrr]^{q_V}
\ar[r];[rd] \ar[rr];[drr] \ar[rrr];[drrr] 
\ar@/_/[r];[rr]_u \ar@{}[d];[dr]|{\subset} \ar@/_1pc/@{-->}[d];[drr]_{\alpha}
& S & G_S & S \\ U' & U & P & V}$$
The global section $u:S\to G_S$ extends to a section $\alpha$ of $q_U$ defined in a neighborhood $S\subset U'$, for instance by linearizing $q_U$ around $S$. Writing $P'=q_U^{-1}(U')$ and  $V'=q_V(P')$, the bibundle $U'\from P'\to V'$ admits a section and is then given by a groupoid map $\phi:G_{U'}\to \nu(G_S)_{V'}$.
This $\phi$ is Morita, and since $\phi^{(0)}|_S$ is invertible and $\d \phi^{(0)}$ is invertible over $S$, then $\phi^{(0)}$ is invertible over some open $S\subset U''\subset U'$. It follows that $\phi|_{G_{U''}}:G_{U''}\to \nu(G_S)_{V''}$ is an isomorphism and therefore a linearization. 
\end{proof}


As an immediate consequence of previous result and the linearization of proper Lie groupoids by exponential maps we get the following:

\begin{proposition}[Stacky Tubular Neighborhood]
A metric on a separated stack $[M/G]$ yields a tubular neighborhood around any embedded substack $[S/G_S]$ by the exponential maps.
\end{proposition}



Note that non-full subgroupoids do not define substacks, and therefore, weak linearization of groupoids does not translate into tubular neighborhoods. 
In \cite{dhdm} we will address the issue of the independence of the stacky exponential map with respect to the groupoid metric representing it. 

Given $\phi:[\tilde M/\tilde G]\to [M/G]$ a stacky surjective submersion, and given $[S/G_S]\subset[M/G]$ a stacky injective immersion, it follows from Lemma \ref{lem:Morita:full:pullback} that the {\bf preimage} $\phi^{-1}([S/G_S])$ is a well-defined substack of $[\tilde M/\tilde G]$, actually represented by the restriction of $\tilde G$ to $\tilde S=\phi^{-1}(S)$,  and that it is embedded  if $[S/G_S]$ is.  

We say that $\phi$ is {\bf linearizable} around $[S/G_S]$ if there are tubular neighborhoods
$\tilde \alpha:[\tilde U/\nu(\tilde G_{\tilde S})_{\tilde U}]\cong[\tilde V/\tilde G_{\tilde V}]$
and
$\alpha:[U/\nu( G_{S})_{ U}]\cong[V/ G_{ V}]$  
such that $\phi\tilde\alpha=\alpha\overline{\d\phi}$.

%

%


\begin{theorem}[Linearization of submersions]
A stacky surjective submersion between separated stacks is linearizable around any embedded substack.
\end{theorem}

\begin{proof}
Without loss of generality, we can suppose that the stacky surjective submersion is represented by a fibration $\phi:\tilde G\to G$ between proper groupoids, and that the embedded substack is represented by a saturated embedded subgroupoid $G_S\toto S$. By Theorem \ref{thm:linear:fib}, $\phi$ is linearizable around $S$, and since $\tilde G$ and $G$ are proper, we can take the groupoid opens $\tilde U$ and $U$ to be full, hence defining stacky tubular neighborhoods and a stacky linearization. 
\end{proof}

We say that a stacky map $\psi/\phi$ is {\em proper} if it can be presented by a fraction on which $\phi$ proper. 

\begin{corollary}[Stacky Ehresmann]
Every stacky proper surjective submersion $\phi/\psi:[\tilde M/\tilde G]\to[M/G]$ between separated stacks admmits a {\em tube} linearization around any embedded substack, namely one on which $[\tilde U/\tilde G_{\tilde U}]=(\overline{\d\phi})^{-1}[U/G_U]$ and $[\tilde V/\tilde G_{\tilde V}]=\phi^{-1}[V/G_V]$.
\end{corollary}

\begin{proof}
We can model the stacky proper surjective submersion with a proper split fibration $\phi:\tilde G\to G$ between proper groupoids, and apply the groupoid version of Ehresmann Theorem \ref{thm:groupoid-ehresmann}.
\end{proof}


These two stacky results generalize and unify many of the linearization and rigidity results about Lie groupoids and related geometries:
\begin{itemize}
\item
When $[\tilde M/\tilde G]$ and $[M/G]$ are manifolds we recover the linearization of submersions and the classical Ehresmann's Theorem, respectively.
\item When $[\tilde M/\tilde G]$ is a manifold, $\psi/\phi$ is a (proper) presentation of $[M/G]$, which is the same as an (s-)proper Lie groupoid $G\toto M$ with orbit stack $[M/G]$. The Linearization Theorem yields the Weinstein-Zung linearization theorem for proper Lie groupoids, while the stacky Ehresmann's Theorem yields the invariant linearization of s-proper Lie groupoids. Special cases of these include the classic linearization of actions, fibrations and foliations. 
\item When $[M/G]$ is a manifold and $\psi/\phi$ is proper we recover the rigidity of compact Lie groupoids (Theorem \ref{thm:rigidity}). This includes as special cases the rigidity for actions, fibrations and foliations (see Section \ref{section:rigidity} and \cite{dhf-note}).
\end{itemize}


{

}



\begin{thebibliography}{xxx}

\bibitem{ac}
C. Arias Abad, M. Crainic;
Representations up to homotopy and Bott's spectral sequence for Lie groupoids.
\emph{Adv.~in Math.~}{\bf 248} (2013), 416--452.

\bibitem{agv}
M. Artin, A. Grothendieck, J. L. Verdier;
Théorie des topos et cohomologie étale des schémas (SGA 4-2);  
Lecture Notes in Mathematics 270; Springer-Verlag (1972). 


\bibitem{bx}
K. Behrend, P. Xu;
Differentiable stacks and gerbes;
Journal of Symplectic Geometry {\bf 9} (2011) 285--341.

\bibitem{bcdh}
H. Bursztyn, A. Cabrera, M. del Hoyo;
Vector bundles over Lie groupoids and algebroids;
\emph{Adv.~in Math.~}{\bf 290} (2016), 163-207;
Preprint \texttt{arXiv:1410.5135}.

\bibitem{bnz}
H. Bursztyn, F. Noseda, C. Zhu:
Principal actions of stacky Lie groupoids. 
International Mathematics Research Notices
Preprint \texttt{arXiv:1510.09208}.



\bibitem{cms}
M.~Crainic, J.~N.~Mestre, I.~Struchiner;
Deformations of Lie Groupoids.
Preprint \texttt{arXiv:1510.02530}.

\bibitem{cs}
M.~Crainic, I.~Struchiner;
On the Linearization Theorem for proper Lie groupoids;
\emph{Ann. Scient. \'Ec. Norm. Sup.} $4^e$ s\'erie, \textbf{46} (2013), 723-746.




\bibitem{dh0}
M. del Hoyo;
On the homotopy type of a cofibred category;
Cahiers de Topologie et Geometrie Differentielle Categoriques 53 (2012), 82--114

\bibitem{dh}
M. del Hoyo;
Lie groupoids and their underlying orbispaces.
\emph{Portugaliae Mathematica} {\bf 70} (2013), 161--209.

\bibitem{dhdm}
M. del Hoyo, M. de Melo;
Geodesics on Differentiable Stacks.
Work in progress.

\bibitem{dhf}
M. del Hoyo, R.L. Fernandes;
Riemannian Metrics on Lie Groupoids;
\emph{Journal f\"ur die reine und angewandte Mathematik (Crelle)}, 2015;
Preprint \texttt{arXiv:1404.5989}.

\bibitem{dhf-note}
M.~del Hoyo, R.L.~Fernandes;
On deformations of compact foliations;
Preprint \texttt{arXiv:1807.10748}.



\bibitem{rosenberg1}
D. Epstein, H. Rosenberg;
Stability of compact foliations. 
In \emph{Geometry and Topology}, Lecture Notes in Mathematics 597, Springer-Verlag (1977), 151--160.

\bibitem{giraud}
J. Giraud;
\emph{Cohomologie non-ab\'elienne}; 
Die Grundlehren der mathematischen Wis\-sen\-scha\-ften, Band 179. Springer-Verlag, Berlin-New York, (1971).





\bibitem{lerman}
E. Lerman;
Orbifolds as stacks? 
\emph{L'Enseign.~Math.~(2)} {\bf 56} (2010), no. 3-4, 315--363.

\bibitem{mkbook}
K. Mackenzie;
\emph{General Theory of Lie Groupoids and Lie Algebroids}.
London Mathematical Society Lecture Note Series 213,
Cambridge University Press (2005).

\bibitem{metzler}
D. Metzler,
Topological and Smooth Stacks.
Preprint \texttt{arXiv:math/0306176}

\bibitem{mm}
I. Moerdijk, J. Mrcun;
\emph{Introduction to Foliations and Lie Groupoids}.
Cambridge Studies in Advanced Mathematics 91,
Cambridge University Press (2003).

\bibitem{mm2}
I. Moerdijk, J. Mrcun;
\emph{Lie groupoids, sheaves and cohomology}.
London Math. Soc. Lecture Note Ser. 323 (2005), 145--272.

\bibitem{nr}
A. Nijenhuis, R. Richardson;
Deformations of homomorphisms of Lie groups and Lie algebras;
Bulletin of the American Mathematical Society 73.1 (1967), 175--179.

\bibitem{pr}
R. Palais, R. Richardson;
Uncountably Many Inequivalent Analytic Actions of a Compact Group on $\R^n$;
Proceedings of the American Mathematical Society 14.3 (1963), 374--377.

\bibitem{palais}
R.~Palais, T.~Stewart;
Deformations of Compact Differentiable Transformation Groups.
\emph{Amer.~J.~of Math.~}{\bf 82} (1960), no. 4, 935--937.




\bibitem{hamilton}
R. Hamilton;
Deformation theory of foliations. 
Available from Cornell University in mimeographed form.


\bibitem{vistoli}
A. Vistoli;
Grothendieck topologies, fibered categories and descent theory; Fundamental algebraic geometry, 1–104, Math. Surveys Monogr., 123, AMS Providence, RI (2005). 

\bibitem{weinstein}
A.~Weinstein;
Linearization of regular proper groupoids. 
\emph{J.~Inst.~Math.~Jussieu} {\bf 1}  (2002), no. 3, 493--511.

\bibitem{zung}
N.T. Zung;
Proper groupoids and momentum maps: linearization, affinity, and convexity. 
\emph{Ann.~Sci.~ \'Ecole Norm.~Sup.~(4)} {\bf 39}  (2006), no. 5, 841--869.

\end{thebibliography}
\end{document}